\documentclass[a4paper,11pt]{article}
\usepackage[top=40mm, bottom=40mm, left=38mm, right=38mm]{geometry}
\usepackage[utf8]{inputenc}
\usepackage[T1]{fontenc}
\usepackage{amsmath,amsthm,amsfonts,amssymb}
\usepackage{enumerate}
\usepackage[colorlinks=true, allcolors=blue]{hyperref}
\usepackage{xcolor, color}
\usepackage{xspace}
\usepackage{comment}
\usepackage[normalem]{ulem}

\theoremstyle{plain}
\newtheorem{theorem}{Theorem}[section]
\newtheorem{lemma}[theorem]{Lemma}
\newtheorem{proposition}[theorem]{Proposition}

\newtheorem*{theorem*}{Theorem}
\newtheorem*{lemma*}{Lemma}
\newtheorem*{proposition*}{Proposition}
\newtheorem*{conjecture*}{Conjecture}
\newtheorem{fact*}{Fact}

\theoremstyle{definition}

\newtheorem{example}[theorem]{Example}
\newtheorem{remark}[theorem]{Remark}
\newtheorem{assumption}[theorem]{Assumption}
\newtheorem*{definition*}{Definition}
\newtheorem*{question*}{Question}
\newtheorem*{example*}{Example}
\newtheorem*{remark*}{Remark}

\numberwithin{equation}{section}

\definecolor{defcolor}{rgb}{0,0,1}
\definecolor{aaltoblue}{cmyk}{100, 39, 0, 0}

\newif\ifshowr
\newif\ifshowchangedi
\newif\ifshowchangedii

\ifshowr
  
\else
  \excludecomment{rcomm}
\fi

\newcommand{\changedi}[1]{{\ifshowchangedi\color{magenta}\fi{#1}}} 
\newcommand{\changedii}[1]{{\ifshowchangedii\color{blue}\fi{#1}}} 
\newcommand{\removedii}[1]{\ifshowchangedii{\color{blue}\sout{#1}}\fi} 

\ifshowchangedi
  \newenvironment{echangedi}
  {\par \color{magenta}}
  {\par}
\else
  \newenvironment{echangedi}
  {\par}
  {\par}
\fi

\ifshowchangedii
  \newenvironment{echangedii}
  {\par \color{blue}}
  {\par}
\else
  \newenvironment{echangedii}
  {\par}
  {\par}
\fi

\newcommand{\mailto}[1]{\href{mailto:#1}{\nolinkurl{#1}}}

\newcommand{\abs}[1]{\mathopen{}\mathclose\bgroup\left|#1\right|\egroup}

\newcommand{\E}{E}
\newcommand{\pr}{P}
\newcommand{\Var}{\operatorname{Var}}

\newcommand{\Cov}{\operatorname{Cov}}

\newcommand{\weq}{\ = \ }

\newcommand{\wle}{\ \le \ }
\newcommand{\wge}{\ \ge \ }

\newcommand{\eqd}{\stackrel{\rm d}{=}}

\newcommand{\prto}{\xrightarrow{\pr}}
\newcommand{\pto}{\prto}

\newcommand{\N}{\mathbb{N}}
\newcommand{\Z}{\mathbb{Z}}

\newcommand{\R}{\mathbb{R}}

\newcommand{\Ito}{It\^{o}\xspace}

\newcommand{\Holder}{H{\"o}lder\xspace}

\newcommand{\Rascanu}{R{\u{a}}{\c{s}}canu\xspace}

\newcommand{\RSint}{\text{\rm (RS)} \int}
\newcommand{\RSsum}{\text{\rm RS}}
\newcommand{\ZSint}{\text{\rm (ZS)} \int}

\newcommand{\Leb}{\operatorname{Leb}}
\newcommand{\ind}[1]{1_{\{#1\}}}
\newcommand{\Hmin}{H_{\rm min}}

\begin{document}

\title{Pathwise Stieltjes integrals of discontinuously evaluated stochastic processes\thanks{A substantial part of this research has been supported by the Emil Aaltonen Foundation.}}
\author{
 Zhe Chen\thanks{New York University, Tandon School of Engineering, Dept.\ Finance and Risk Engineering} \and
 Lasse Leskelä\thanks{
 Aalto University, School of Science,
 Dept.\ Mathematics and Systems Analysis
 }
 \and
 Lauri Viitasaari\thanks{Aalto University, School of Electrical Engineering,
 Dept.\ Communications and Networking
 }
}
\date{\today}
\maketitle

\begin{abstract}
In this article we study the existence of pathwise Stieltjes integrals of the form $\int f(X_t)\, dY_t$ for nonrandom, possibly discontinuous, evaluation functions $f$ and \Holder continuous random processes $X$ and $Y$. We discuss a notion of sufficient variability for the process $X$ which ensures that the paths of the composite process $t \mapsto f(X_t)$ are almost surely regular enough to be integrable. We show that the pathwise integral can be defined as a limit of Riemann--Stieltjes sums for a large class of discontinuous evaluation functions of locally finite variation, and provide new estimates on the accuracy of numerical approximations of such integrals, together with a change of variables formula for integrals of the form $\int f(X_t) \, dX_t$.
\end{abstract}

{
}

\section{Introduction}

In this article we study the existence of pathwise Stieltjes integrals of the form
\begin{equation}
 \label{eq:int}
 \int_0^T f(X_t) \, dY_t,
\end{equation}
where $X$ and $Y$ are real-valued stochastic processes with time parameter $t \in [0,T]$, and the values of $X$ are evaluated using a nonrandom function $f: \R \to \R$. If $X$ is adapted and $Y$ is a semimartingale with respect to some common filtration, then the integral can be handled using classical \Ito calculus techniques. On the other hand, if the function $f$ and the paths of $X$ and $Y$ are \Holder continuous of sufficiently \changedi{high} order, then the above integral exists almost surely as a standard Riemann--Stieltjes integral, as a consequence of the classical results of L.~C.~Young~\cite{Dudley_Norvaisa_1998, LCY}. However, in \changedi{several application contexts, especially fractional-order nonlinear systems \cite{Monje_Chen_Vinagre_Xue_Feliu_2010, Sheng_Chen_Qiu_2012},}  it is important to consider stochastic processes which are not semimartingales and evaluation functions which are not continuous. For example, in bang-bang stochastic control \cite{Kushner_Dupuis_1992} and stop-loss trading \cite{Follmer_Schied_2004, Sondermann} the function $f$ corresponds to the indicator function of a certain threshold. In this case the paths of $t \mapsto f(X_t)$ are typically discontinuous, and even worse, of infinite $p$-variation for all $p \ge 1$, which rules out rough path techniques introduced in \cite{Friz_Hairer_2014, lyons}.

Motivated by stochastic models where $X$ and $Y$ are non-semimartingales and $f$ is discontinuous, we investigate conditions which allow to define the integral~\eqref{eq:int} pathwise. The key idea in our approach is the assumption that \emph{sufficient variability} in the probability distribution  of $X$ makes it unlikely for $X$ to spend much time in the discontinuity points of $f$ and renders the composite process $t \mapsto f(X_t)$ regular enough to be integrable. Our main contribution is to quantify the notion of sufficient variability and show that the integral~\eqref{eq:int} can be defined as a limit of Riemann--Stieltjes sums for a large class of discontinuous evaluation functions of locally finite variation. We also provide new estimates on the accuracy of numerical approximations of such integrals computed using Riemann--Stieltjes sums over finite dense partitions, together with a change of variables formula for integrals of the form $\int f(X_t) \, dX_t$.

The main technical tool in our analysis is the pathwise fractional Stieltjes integral introduced by Zähle in \cite{Zahle_1998}, which is based on reinterpreting the Stieltjes integral as the Lebesgue integral of a product of fractional derivatives. The main technical part of the article is to show that under sufficient regularity conditions, the paths of $f(X_t)$ have almost surely finite Gagliardo seminorms and fit into fractional Sobolev spaces where fractional derivatives behave well \cite{DiNezza_Palatucci_Valdinoci_2012}. Consequently, the integral \eqref{eq:int} exists as a fractional Stieltjes integral. With this result at hand, we show how ordinary Riemann--Stieltjes sums approximating \eqref{eq:int} can also be interpreted as fractional Stieltjes integrals, and the convergence of Riemann--Stieltjes sums is then obtained by showing that suitably interpolated paths of $f(X_t)$ converge to the actual paths in the corresponding fractional Sobolev space.

Since the seminal paper of Zähle~\cite{Zahle_1998}, fractional Stieltjes integrals of stochastic processes have been actively studied, especially for fractional Brownian motions. Nualart and \Rascanu \cite{Nualart_Rascanu_2002} used fractional calculus techniques to study differential equations driven by \Holder continuous processes; all integrals of the form \eqref{eq:int} were restricted to Lipschitz continuous evaluation functions $f$. To the best of our knowledge, the first paper to discuss integrals of the form \eqref{eq:int} for discontinuous evaluation functions is \cite{AMV}, where similar types of questions as in this paper were studied for geometric fractional Brownian motions. Unfortunately, some proofs of the key results in \cite{AMV} contain flaws that appear difficult to fix without incorporating additional assumptions. Tikanmäki \cite{Tikanmaki_2002} extended the analysis to nonlinear integral functionals of fractional Brownian motions.
Another closely related paper is \cite{A-V}, where the rate of convergence of forward Riemann--Stieltjes sums over uniform partitions were studied for fractional Brownian motions. All of the aforementioned papers and most others focus on fractional Brownian motions of Hurst index $H>1/2$. However, see 
\cite{Besalu_Nualart_2011} for a special case with $1/3 < H < 1/2$.

In contrast with the aforementioned works, our current paper studies the integrability questions in a general framework with minimal assumptions, not restricted to fractional Brownian motions nor even Gaussian processes. Besides being more general, the key results, when restricted to fractional Brownian motions, turn out also to be stronger (see Example~\ref{exa:fBm}).



The rest of the paper is organised as follows. Section \ref{sec:assumptions} introduces the key assumptions, Section~\ref{sec:results} contains the main results, and Section~\ref{sec:proofs} contains the proofs. Some of the more technical proofs are deferred into separate appendices in the end of the paper.

\section{Key assumptions}
\label{sec:assumptions}
\subsection{Sufficient variability}

We make the following assumption \changedi{on} a random process $X:[0,T] \to \R$, which guarantees that the probability distribution of $X_t$ is not too much concentrated around any single value.

\begin{assumption}
\label{ass:Density}
For almost every \changedi{$t \in (0,T)$}, the random variable $X_t$ admits a probability density function $x \mapsto p_t(x)$ that is bounded according to
\begin{equation}
 \label{eq:Density}
 \sup_{x \in \R} p_t(x)
 \wle \hat p_t,
\end{equation}
where $\int_0^T \hat p_t \, dt < \infty$.
\end{assumption}
This assumption holds for a broad variety of random processes, for example:
\begin{itemize}
\item If $X$ is a stationary process such that $X_0$ has a bounded probability density function, then \eqref{eq:Density} holds with $\hat p_t$ being a constant.
\item If $X$ is a Gaussian random process with autocovariance function $R(s,t) = \Cov(X_s,X_t)$ such that $\int_0^T R(t,t)^{-1/2} \, dt < \infty$, then \eqref{eq:Density} holds with $\hat p_t = R(t,t)^{-1/2}$. Especially, Brownian motions and all fractional Brownian motions with Hurst parameter $H \in (0,1)$ satisfy Assumption~\ref{ass:Density}.
\item
\changedii{Let $X$ be a self-similar random process with index $H \in(0,1)$ and such that $X_1$ has a a bounded probability density function $p_1$.
Then $X_t$ has probability density function $p_t(x) = t^{-H} p_1(t^{-H}x)$ for $t > 0$. Hence
\[
 \int_0^T \sup_x p_t(x) \, dt
 \weq \Big( \sup_x p_1(x) \Big) \int_0^T t^{-H} \, dt
 \ < \ \infty
\]
shows that Assumption \ref{ass:Density} is satisfied.}
\end{itemize}
\changedi{
Observe also that \changedii{if} $X$ and $Y$ are independent random processes such that $X$ satisfies Assumption~\ref{ass:Density}, then $Z_t = X_t+Y_t$ admits a probability density function $z \mapsto \E p_t(z-Y_t)$ for almost every $t$, and hence also the process $Z=X+Y$ satisfies Assumption~\ref{ass:Density}.
}


\subsection{Continuity}
Most of the underlying random process studied in this paper will be assumed to have almost surely \Holder continuous paths. In addition, \changedii{we shall often assume that the process $X$ satisfies
\[
 \sup_{0 \le s < t \le T} \frac{\left( \E \abs{X_t - X_s}^{r} \right)^{1/r}}{|t-s|^{H}} < \infty
\]
for some $H \in (0,1]$ and $r > 1/H$,
}
which is the classical Kolmogorov--Chentsov criterion \cite[Theorem 3.23]{Kallenberg_2002} for the existence of a continuous version of $X$ with almost surely $\alpha$-\Holder continuous paths for any \changedii{$\alpha < H - 1/r$}.


\begin{remark}
\label{rem:GaussianHolder}
Any Gaussian process with \Holder continuous paths of order $\alpha>0$ satisfies \changedii{the above condition
for any $r > 1$ and any $H < \alpha$}, see \cite{asvy14}.
\end{remark}

\section{Results}
\label{sec:results}

\subsection{Pathwise integrability}
Throughout this section we denote by $\ZSint$ the fractional Zähle--Stieltjes integral which \changedi{is described in detail} in Section~\ref{subsec:fractional_integrals}.

\begin{theorem}
\label{the:ZSIntegral}
Let $X$ and $Y$ be \Holder continuous random processes of orders $\alpha$ and $\beta$, respectively, such that $\alpha + \beta > 1$. \changedi{Assume that $X$ satisfies Assumption~\ref{ass:Density}}. Then for any $f: \R \to \R$ of locally finite variation such that $f(X_{0+})$ exists almost surely, the pathwise integral
\[
 \ZSint_0^T f(X_t) \, d Y_t
\]
exists almost surely.
\end{theorem}

\begin{remark}
\changedi{The limit $f(X_{0+}) = \lim_{t \to 0} f(X_t)$ exists on the event that $X_0$ lies outside the discontinuity set of $f$. Especially,  the assumption that $f(X_{0+})$ exists almost surely is automatically valid when $X_t$ admits a probability density function at $t=0$, because the set of discontinuities of $f$ is countable.
}
\end{remark}

\begin{remark}
\label{rem:sufficient}
By examining the proof of Theorem~\ref{the:ZSIntegral} (actually, its key argument in the proof of Proposition~\ref{the:GagliardoComposite}) it is evident that a sufficient condition for the existence of the integral is that
\begin{equation}
 \label{eq:sufficient}
 \sup_{y \in \R} \E \int_0^T |X_t - y|^{-\theta/\alpha} \, dt < \infty
\end{equation}
for some $\theta \in (1-\beta,\alpha)$. This condition means that while $X$ may (and does in our cases of interest) cross some levels infinitely often, the crossings do not occur too often in the above sense. Furthermore, Assumption~\ref{ass:Density} guarantees that \eqref{eq:sufficient} holds for \changedi{the} process $X$. Consequently, the integral exists for all \Holder continuous nonrandom processes $x: [0,T] \to \R$ satisfying
\[
 \sup_{y \in \R} \int_0^T |x_t-y|^{-\theta/\alpha} \, dt < \infty.
\]
\end{remark}

\subsection{Change of variables}

\begin{theorem}
\label{the:Ito}
Assume that $X$ is \Holder continuous \changedi{of order $\alpha > 1/2$ and satisfies Assumption~\ref{ass:Density}.}
Let $f: \R \to \R$ be absolutely continuous, having a derivative $f'$ of locally finite variation such that $f'(X_{0+})$ exists almost surely. Then for any $t \in [0,T]$,
\begin{equation}
 \label{eq:Ito}
 f(X_t) - f(X_0)
 \weq \ZSint_0^t f'(X_s) \, dX_s
\end{equation}
almost surely.
\end{theorem}

\subsection{Riemann--Stieltjes integrability}

A tagged partition of $[0,T]$ is a sequence $\pi = (t_0,\dots, t_k; \xi_1,\dots,\xi_k)$ such that $0=t_0 < t_1 < \cdots < t_k=T$ and $\xi_i \in [t_{i-1},t_i]$ for all $i$. The mesh of $\pi$ is denoted by $||\pi|| = \max_i (t_i-t_{i-1})$, and the Riemann--Stieltjes sum of $f$ against $g$ along $\pi$ by
\[
 \RSsum(f,g,\pi)
 \weq \sum_{i=1}^k f(\xi_i) (g(t_i)-g(t_{i-1})).
\]
If $\RSsum(f,g,\pi_n)$ converges to a unique limit for any sequence of tagged partitions such that $||\pi_n|| \to 0$, then this limit is denoted $\RSint_0^T f_t d g_t$ and called the Riemann--Stieltjes integral of $f$ against $g$.


The next result concerns the Riemann--Stieltjes integrability of random processes and can be viewed as the main result of this paper. We denote convergence in probability by $\prto$.

\begin{theorem}
\label{the:RSIntegral}
\changedii{
Let $X$ be a stochastic process such that
$X_t$ admits a probability density function $x \mapsto p_t(x)$ for almost every $t \in (0,T)$ which is bounded according to Assumption~\ref{ass:Density} and satisfies
\begin{equation}
 \label{eq:DensityBound}
 \sup_{\abs{x} \ge \epsilon/2} \! p_t(x) \wle c_0
\end{equation}
for some constants $\epsilon \ge 0$ and $c_0 < \infty$; and
\begin{equation}
 \label{eq:MeanContinuity}
 \sup_{0 \le s < t \le T} \frac{\left(\E \abs{X_t - X_s}^{r}\right)^{1/r}}{\abs{t-s}^{H}} < \infty
\end{equation}
for some $H \in (0,1]$ and $r > 1/H$.  Let $f: \R \to \R$ be of locally finite variation such that $f(X_{0+})$ exists almost surely and $f$ restricted to $(-\epsilon,\epsilon)$ is Lipschitz continuous. Then for any \Holder continuous process $Y$ of order $\beta > 1-(H-1/r)$} there exists a random variable $\RSint_0^T f(X_t) \, dY_t$ such that
\[
 \RSint_0^T f(X_t)\, d Y_t
 \weq \ZSint_0^T f(X_t)\, d Y_t
 \qquad \text{a.s.}
\]
and
\begin{equation}
 \label{eq:RSLimit}
 \RSsum(f \circ X,Y,\pi_n)
 \ \prto \ \RSint_0^T f(X_t)\, d Y_t
\end{equation}
for any sequence of tagged partitions with $||\pi_n|| \to 0$. Furthermore, the convergence in \eqref{eq:RSLimit} holds almost surely for any sequence of tagged partitions satisfying 
$
 \sum_{n=1}^\infty ||\pi_n||^a < \infty
$
for some \changedii{$a < \frac{r}{1+r} H - (1-\beta)$}.
\end{theorem}

\changedi{
In contrast with the earlier results regarding fractional Zähle--Stieltjes integrals, in Theorem \ref{the:RSIntegral} we need to impose slightly stronger conditions than Assumption~\ref{ass:Density} on the sufficient variability of $X$, to  be able to prove the convergence of Riemann--Stieltjes sums. \changedii{Condition~\eqref{eq:DensityBound} with $\epsilon = 0$ essentially means that the probability density function of $X_t$ is uniformly bounded over $t \in (0,T)$ and $x \in \R$.
This condition is valid e.g.\ when $X$ is a stationary processes with a bounded density, but typically fails for random processes with stationary increments such that $X_0 = 0$ almost surely. Condition~\eqref{eq:DensityBound} with $\epsilon > 0$ relaxes the sufficient variability assumption on $X$, and compensates this by requiring $f$ to be Lipschitz continuous in a neighborhood of $X_0=0$. See Section~\ref{sec:Examples} for examples and further discussion.
}
}

\subsubsection{Dyadic partitions}
The following result considers the \changedi{convergence rate of Riemann--Stieltjes approximations over dyadic partitions.
\begin{theorem}
\label{the:RateDyadic}
Assume that the triplet $(X,Y, f)$ satisfies the assumptions of Theorem \ref{the:RSIntegral}, and let
$\pi_n$ be a tagged partition of $[0,T]$ with uniform mesh of size $T 2^{-n}$. Then \changedii{for any $\epsilon>0$} there exists an almost surely finite random variable $C(\omega)$ such that
\[
 \abs{\RSsum(f \circ X,Y,\pi_n) - \RSint_0^T f(X_t)\, d Y_t}
 \wle \changedii{C 2^{-\left( \frac{r}{1+r} H - (1-\beta) - \epsilon\right) n}}
\]
\changedii{almost surely for all $n$}.
\end{theorem}
}

\subsubsection{Accuracy of interpolations}

For the following result, we need to rule out functions of locally finite variation for which the variation $V_{[-k,k]}(f)$ over the interval $[-k,k]$ grows too fast as $k \to \infty$. A sufficient condition is to assume that the growth rate of $V_{[-k,k]}(f)$ is slow compared to the decay rate of the tail probability distribution of the maximum of $X$ over $[0,T]$, in the sense that
\begin{equation}
 \label{eq:full_support}
 \sum_{k=1}^\infty V_{[-k,k]}(f) \, \pr(||X||_\infty > k-1)^{\frac{1}{p}}
 \ < \ \infty
\end{equation}
for all $p \ge 1$. 

\begin{theorem}
\label{the:RateMean}
\changedi{Assume that the triplet $(X,Y, f)$ satisfies the assumptions of Theorem \ref{the:RSIntegral}, and that the \Holder seminorm $[Y]_{\beta,\infty}$ of $Y$ is bounded by $\E [Y]_{\beta,\infty}^{{p}} < \infty$ for all ${p} \in [1,\infty)$}. Assume further that either $f$ has finite variation, or $f$ satisfies \eqref{eq:full_support}. Then for any \changedii{$a < \frac{r}{1+r} H - (1-\beta)$} there exists a finite constant $c$ such that
\begin{equation}
 \label{eq:L_rate}
 \E \, \abs{ \RSsum(f \circ X,Y,\pi) - \RSint_0^T f(X_t)\, d Y_t }
 \wle c ||\pi||^a
\end{equation}
for all tagged partitions of $[0,T]$ with mesh size $||\pi|| \le 1$.
\end{theorem}

\begin{remark}
\changedi{The condition on the \Holder seminorm $[Y]_{\beta,\infty}$ in Theorem~\ref{the:RateMean} is satisfied by all Gaussian processes \cite{asvy14} that are \Holder continuous of order strictly larger than $\beta$.}
\end{remark}


\subsection{Examples}
\label{sec:Examples}


\subsubsection{Gaussian processes}

\begin{example}[Stationary Gaussian process]
Let $X$ be a stationary $\alpha$-\Holder continuous Gaussian process \changedii{with zero mean and nonzero variance}, e.g.\ a fractional Ornstein--Uhlenbeck process \cite{Cheridito_Kawaguchi_Maejima_2003}. \changedii{Then $X$ automatically satisfies Assumption~\ref{ass:Density} and condition \eqref{eq:DensityBound} of Theorem~\ref{the:RSIntegral} with $\epsilon = 0$. Moreover, Remark~\ref{rem:GaussianHolder} implies that $X$ satisfies condition \eqref{eq:MeanContinuity} of Theorem~\ref{the:RSIntegral} for any $H \in (0,\alpha)$ and $r > 1$. If $f: \R \to \R$ is of locally finite variation, then $f(X_{0+})$ exists almost surely because the set of points at which $f$ is discontinuous is countable. Hence by Theorem~\ref{the:RSIntegral} we may conclude that $t \mapsto f(X_t)$ is almost surely Riemann--Stieltjes integrable against any \Holder continuous process $Y$ of order $\beta > 1 - \alpha$, and Theorem~\ref{the:RateMean} provides error estimates for numeric integration under sufficient additional regularity for $f$ and $Y$.
}
%
\end{example}

\begin{example}[Fractional Brownian motion]
\label{exa:fBm}
Let $X$ and $Y$ be (possibly dependent) fractional Brownian motions with Hurst indices $H_1, H_2 \in (0,1)$, respectively, such that $H_1 + H_2 > 1$, and let $f$ be of locally finite variation and Lipschitz continuous in a neighborhood $(-\epsilon,\epsilon)$ of zero. \changedii{Denote by $\phi(x; \mu, \sigma)$ the Gaussian probability density with mean $\mu$ and standard deviation $\sigma$, and observe that
\begin{equation}
 \label{eq:GaussianDensity1}
 \max_{x \in \R} \phi(x; \mu, \sigma)
 \weq (2\pi)^{-1/2} \sigma^{-1},
\end{equation}
and that the inequality $u e^{-u} \le e^{-1}$ implies
\begin{equation}
 \label{eq:GaussianDensity2}
 \phi(x; 0, \sigma)
 \wle (2/\pi)^{1/2} e^{-1} \sigma x^{-2}.
\end{equation}
Because the distribution of $X_t$ has density $p_t(x) = \phi(x; 0, t^{H_1})$ for all $t > 0$, \eqref{eq:GaussianDensity1} implies that $X$ satisfies Assumption \ref{ass:Density}. Moreover, \eqref{eq:GaussianDensity2} implies that $p_t(x) \le (2/\pi)^{1/2} e^{-1} T^{H_1} (\epsilon/2)^{-2}$ for all $|x| \ge \epsilon/2$ and all $t \in (0,T)$, and the fact that $X$ is self-similar of order $H_1$ and has stationary increments shows that $\left(\E \abs{X_t-X_s}^r\right)^{1/r} = \left(\E \abs{X_1}^r\right)^{1/r} \abs{t-s}^{H_1}$ for all $r > 0$. Hence $X$ satisfies conditions  \eqref{eq:DensityBound} and \eqref{eq:MeanContinuity} of Theorem~\ref{the:RSIntegral}.

Given an arbitrary $0 < a < H_1 + H_2 - 1$, choose $\beta \in (0,H_2)$ close enough to $H_2$ and $r>1$ large enough so that $a < \frac{r}{1+r}H_1 + \beta - 1$ and $H_1 + \beta - 1 > 1/r$.} Furthermore, the supremum of a Gaussian process has Gaussian tails \cite{sam}, and consequently \eqref{eq:full_support} is satisfied for all $f$ for which $V_{[-k,k]}(f)$ has at most exponential growth. Finally, because $Y$ is a Gaussian process with \Holder continuous paths of order $\beta$, then $\E [Y]_{\beta,\infty}^q < \infty$ for all $q \ge 1$ \cite{asvy14}. Hence the triplet $(X,Y,f)$ fulfills all assumptions required in Theorems~\ref{the:RSIntegral} and \ref{the:RateMean}, and we may conclude that
\[
 \E \, \left| \RSsum(f \circ X,Y,\pi) - \RSint_0^T f(X_t)\, d Y_t \right|
 \wle c ||\pi||^{H_1 + H_2 - 1 - \epsilon}
\]
for all $\epsilon =  H_1 + H_2 - 1 - a > 0$. \changedi{We believe that this convergence rate is optimal because it coincides with a similar error estimate for smooth functions.} A similar type of result for $H_1=H_2 > 1/2$, restricted to forward Riemann--Stieltjes sums over a uniform mesh, was discovered in \cite[Theorem 3.1]{A-V}, but with a significantly worse rate of $||\pi||^{H-1/2-\epsilon}$. \changedi{The improved error rate is a consequence of a sharper estimate (Lemma \ref{the:LipschitzCdf}) compared to \cite[Lemma 3.1]{A-V}.}
\end{example}

\begin{echangedii}
\begin{example}[Multifractional Brownian motion]
Consider a multifractional Brownian motion $X$ parametrized by an $\alpha$-Hölder continuous function $H: [0,T] \to (0,1)$, so that
\[
 X_t
 \weq \int_{-\infty}^\infty \Big( (t-u)_+^{H(t)-1/2} - (-u)_+^{H(t)-1/2} \Big) \, W(du),
\]
where $W$ denotes standard Gaussian white noise on $\R$. The variance of this zero-mean Gaussian process can be expressed \cite[Lemma 4.1]{Stoev_Taqqu_2006} as
\[
 \Var X_t
 \weq 2 t^{2H(t)} \frac{\Gamma(H(t)+1/2)^2}{\pi} \int_0^\infty \frac{1-\cos(\xi)}{\xi^{2H(t)+1}} d \xi.
\]
Because
\[
 \int_0^\infty \frac{1-\cos(\xi)}{\xi^{2H(t)+1}} d \xi
 \wge \int_{\pi/2}^{3\pi/2} \frac{1}{\xi^{2H(t)+1}} d \xi
 \wge \int_{\pi/2}^{3\pi/2} \frac{1}{\xi^{3}} d \xi
 \wge \int_{\pi/2}^{\infty} \frac{1}{\xi^{3}} d \xi
 \weq 2 \pi^{-2},
\]
it follows that $\Var X_t \ge c t^{2H(t)}$ with $c = 4 \pi^{-3} \min_{x \ge 1/2} \Gamma(x)$. Because $X_t$ has a Gaussian distribution and $H(t)$ takes values in a compact subinterval of $(0,1)$,
we see with the help of \eqref{eq:GaussianDensity1} that $X$ satisfies Assumption \ref{ass:Density}.
One may also verify that $\sup_{t \in [0,T]} \Var X_t < \infty$, and hence with the help of \eqref{eq:GaussianDensity2} we see that $X$ satisfies condition \eqref{eq:DensityBound} of Theorem~\ref{the:RSIntegral} for any $\epsilon > 0$.
Furthermore, the analysis in \cite{levy-vehel_peltier-1995} implies that $X$ satisfies condition~\eqref{eq:MeanContinuity} of Theorem~\ref{the:RSIntegral} with exponent $\min(\alpha, \Hmin, 1/2)$ in the denominator for any $r \ge 1$, where $\Hmin = \min_{t \in [0,T]} H(t)$. Hence the conclusions of Theorems~\ref{the:RSIntegral}--\ref{the:RateMean} are valid for all sufficiently regular $f$ and $Y$.
\end{example}
\end{echangedii}

\subsubsection{Non-Gaussian processes}

\begin{echangedii}
\begin{example}[fBm-driven SDEs]
Suppose that $X$ is a solution to the SDE 
\[
 X_t
 \weq \int_0^t V_0(X_s) ds + \sum_{i=1}^d \int_0^t V_i(X_s)d B_s^{i},
\]
where $B^{1}, \dots, B^d$ are independent fractional Brownian motions with Hurst parameter $H \in (\frac12, 1)$ and $V_0, \dots, V_d$ are bounded infinitely differentiable functions on the real line such that
$\inf_x \sum_{i=1}^d V_i(x)^2 > 0$.
Choose an arbitrary $\alpha \in (1-H,\frac12)$. Then, by \cite[Theorem 2.1]{Nualart_Rascanu_2002}, the solution exists and is unique in the space of stochastic processes satisfying
\[
 \Vert X \Vert_{\alpha,\infty}
 := \sup_{t\in[0,T]}\left(|X_t| + \int_0^t \frac{|X_t-X_s|}{(t-s)^{\alpha+1}}ds\right) < \infty
\] 
almost surely, and $\E \Vert X \Vert_{\alpha,\infty}^r < \infty$ for all $r \ge 1$. By combining Propositions 4.2 and 4.4 of \cite{Nualart_Rascanu_2002} together with the fact $\Vert f \Vert_\infty \le \Vert f \Vert_{\alpha,\infty}$, we observe that the \Holder seminorm of $X$ is bounded by
\[
 [X]_{1-\alpha,\infty}
 \wle C (\Lambda_\alpha(B)+1)(1+\Vert X\Vert_{\alpha,\infty}),
\]
where, by \cite[Lemma 7.5]{Nualart_Rascanu_2002}, the random variable $\Lambda_\alpha(B)$ satisfies $\E |\Lambda_\alpha(B)|^r < \infty$ for all $r \ge 1$. This implies that for all $r \ge 1$,
\[
 \sup_{0 \le s<t \le T}\frac{\left(\E|X_t - X_s|^r\right)^{1/r}}{|t-s|^{1-\alpha}}
 < \infty.
\]
Furthermore, by \cite[Theorem 1.5]{baudoin_nualart_ouyang_tindel-2016}, the solution admits a probability density function $p_t$ satisfying 
\[
 p_t(x)
 \wle c_1t^{-H}e^{-\frac{x^2}{c_2t^{2H}}}
\]
for some constants $c_1,c_2>0$. Consequently, $X$ satisfies Assumption~\ref{ass:Density} and conditions \eqref{eq:DensityBound}--\eqref{eq:MeanContinuity} of Theorem~\ref{the:RSIntegral} for any given $\epsilon > 0$. Hence we may apply Theorems~\ref{the:RSIntegral}--\ref{the:RateMean} for all sufficiently regular $f$ and $Y$.
\end{example}
\end{echangedii}

\begin{echangedii}
\begin{example}[Rosenblatt process]
A Hermite process of order $n \ge 1$ with Hurst parameter $H \in (1/2,1)$ is defined as a multiple Wiener--\Ito integral
\[
 X_t
 \weq \int_{\R^n} \int_0^t \left( \prod_{k=1}^n (s-u_k)_+^{- \left( \frac12  + \frac{1-H}{k} \right) } \right) ds \, dB(u_1) \cdots dB(u_n),
\]
where $B$ is a standard Brownian motion. This process is $H$-self-similar and has stationary increments, but is not Gaussian expect in the special case $n=1$ where it reduces to a fractional Brownian motion~\cite{maejima_tudor-2007}. The case $n=2$ is known as the Rosenblatt process, and in this case $X_1$ has finite moments of all orders and a bounded infinitely differentiable density function \cite{Veillette_Taqqu_2013}. Self-similarity implies that $p_t(x) = t^{-H}p_1(t^{-H}x)$. Moreover, because $\E \abs{X_1}$ is finite, it follows that $p_t(x) \le ( \sup_u \abs{u} p_1(u)) \abs{x}^{-1}$ for all $x \ne 0$.
Furthermore, self-similarity and stationary increments imply that
$
 \left( \E \abs{X_t - X_s}^{r} \right)^{1/r}
 = \left(\E \abs{X_1}^r \right)^{1/r} \abs{t-s}^{H}
$
for any $r > 0$. These properties imply that $X$ satisfies Assumption~\ref{ass:Density} and conditions \eqref{eq:DensityBound}--\eqref{eq:MeanContinuity} of Theorem~\ref{the:RSIntegral} for any $\epsilon > 0$. Hence the conclusions of Theorems~\ref{the:RSIntegral}--\ref{the:RateMean} are valid for all sufficiently regular $f$ and $Y$.

\end{example}
\end{echangedii}

\begin{echangedii}
\begin{example}[Iterated fractional Brownian motion]
Consider an iterated fractional Brownian motion defined by
\[
 X_n(t)
 \weq B_n \Big( \cdots \big( B_3 \big( B_2 \big(B_1(t) \big) \big) \big) \cdots \Big),
\]
where $B_1,\dots, B_n$ are independent two-sided fractional Brownian motions with Hurst indices $H_1,\dots, H_n$. When 
$H_i = 1/2$ for all $i$, this process reduces to the iterated Brownian motion, a process which has been actively studied  \cite{burdzy-1993,Casse_Marckert_2016,Curien_Konstantopoulos_2014}.
Conditioning shows that
\begin{equation}
 \label{eq:IteratedFirst}
 X_n(t) - X_n(s)
 \ \eqd \ \abs{X_{n-1}(t) - X_{n-1}(s)}^{H_n} B_n(1),
\end{equation}
and by repeating this we obtain
\begin{equation}
 \label{eq:Iterated}
 X_n(t) - X_n(s)
 \ \eqd \ \abs{t-s}^{H} \left( \prod_{k=1}^{n-1} \abs{B_k(1)}^{ \frac{H_1 \cdots H_n}{H_1 \cdots H_k} } \right)B_n(1)
\end{equation}
with $H = H_1 \cdots H_n$.

Equation \eqref{eq:IteratedFirst} with $s=0$ shows that the conditional distribution of $X_n(t)$ given $X_{n-1}(t)$ is Gaussian with mean zero and standard deviation $|X_{n-1}(t)|^{H_n}$, so that the unconditional distribution of $X_n(t)$ has density function
\[
 p_t(x)
 \weq \E \phi(x; 0, |X_{n-1}(t)|^{H_n}),
\]
where $\phi(x; \mu, \sigma)$ denotes Gaussian probability density with mean $\mu$ and standard deviation $\sigma$. We see by applying \eqref{eq:Iterated} with $s=0$ that
\[
 p_t(0)
 \weq (2\pi)^{-1/2} \E |X_{n-1}(t)|^{-H_n}
 \weq (2\pi)^{-1/2} \left(\prod_{k=1}^{n-1} \E |B_k(1)|^{- \frac{H_1 \cdots H_n}{H_1 \cdots H_k} }\right) |t|^{-H}.
\]
Because $p_t(x)$ attains its maximum at $x=0$ and the right side above is integrable over $[0,T]$, we conclude that $X$ satisfies Assumption \ref{ass:Density}. Observe next with the help of \eqref{eq:GaussianDensity2} that
\begin{align*}
 p_t(x)
 &\wle (2/\pi)^{1/2} e^{-1}  x^{-2} \E |X_{n-1}(t)|^{H_n} \\
 &\wle (2/\pi)^{1/2} e^{-1}  (\epsilon/2)^{-2} \left(\prod_{k=1}^{n-1} \E |B_k(1)|^{\frac{H_1 \cdots H_n}{H_1 \cdots H_k} }\right) |T|^H
\end{align*}
for all $|x| > \epsilon/2$ and $t > 0$. Hence $X$ satisfies condition \eqref{eq:DensityBound} of Theorem~\ref{the:RSIntegral}. A multivariate version of \eqref{eq:Iterated} shows that $X_n$ is self-similar of order $H$, as was found for the Brownian case in \cite{Curien_Konstantopoulos_2014}. Equation \eqref{eq:Iterated} also implies that
\[
 \left( \E |X_n(t) - X_n(s)|^r \right)^{1/r}
 \weq \left( \prod_{k=1}^{n} \E \abs{B_k(1)}^{ \frac{H_1 \cdots H_n}{H_1 \cdots H_k} r } \right) ^{1/r} \abs{t-s}^{H}
\]
for all $r > 0$. Hence $X$ also satisfies condition \eqref{eq:MeanContinuity} of Theorem~\ref{the:RSIntegral}. Hence the conclusions of Theorems~\ref{the:RSIntegral}--\ref{the:RateMean} are valid for all sufficiently regular $f$ and $Y$.
\end{example}
\end{echangedii}

\begin{echangedii}
\begin{example}[Weak Brownian motion]
Let $X$ be a weak Brownian motion~\cite{follmer_wu_yor-2000} of order $k \ge 2$ which has the same $k$-dimensional marginal distributions as a standard Brownian motion, but is not a Brownian motion. Because the one and two dimensional marginal distributions coincide with those of the standard Brownian motion, it follows that $X_t$ follows a Gaussian distribution with mean zero and standard deviation $t^{1/2}$, and
\[
 \left( \E \abs{X_t - X_s}^{r} \right)^{1/r}
 \weq \left( \E \abs{X_1}^{r} \right)^{1/r} \abs{t-s}^{1/2}.
\]
Consequently, we may apply the same arguments as in Example~\ref{exa:fBm} to conclude that $X$ satisfies Assumption~\ref{ass:Density} and conditions~\eqref{eq:DensityBound} and~\eqref{eq:MeanContinuity} of Theorem~\ref{the:RSIntegral} with $H=1/2$ and arbitrary $r > 2$ and $\epsilon > 0$. 
\end{example}
\end{echangedii}

\begin{echangedi}
\begin{example}[Financial portfolio]
Consider a stochastic process of the form
\[
 S_t = S_0e^{\sigma X_t - g(t,\sigma)}
\]
where $\sigma > 0$ is a constant, $t \mapsto g(t,\sigma)$ is an $\alpha$-\Holder continuous drift function, and 
$X$ is an $\alpha$-\Holder continuous random process such that $\sup_{t\in[0,T]}|X_t|$ has all exponential moments. When $\alpha > \frac12$ and $f$ is any function that satisfies the assumptions of Theorem \ref{the:RSIntegral}, then
%
%
\[
\int_0^T f(S_t)\, d S_t
\]
exists as a Riemann--Stieltjes integral. Here $S_t$ can be interpreted as a stock price, and the integral as the value of a portfolio with $f(S_t)$ as the amount of stock at time $t$. Because $\sup_{t\in[0,T]}|X_t|$ has exponential moments, we may use Theorem \ref{the:RSIntegral} to obtain 
\[
 \E \, \left| \RSsum(f \circ S,S,\pi) - \RSint_0^T f(S_t)\, d S_t \right|
 \wle c ||\pi||^{\frac{r}{1+r}H + \beta - 1 - \epsilon}.
\]
The Riemann--Stieltjes sum $\RSsum(f \circ S,S,\pi)$ can be interpreted as the value of a hedging portfolio constructed by approximating a continuous trading strategy.
\end{example}
\end{echangedi}

\begin{echangedi}
\begin{example}[Fractional splines]
Fractional splines are $\alpha$-\Holder continuous functions which can be represented in the form
\begin{equation}
 \label{eq:Spline}
 x_t \weq \sum_{k \in \Z} a_k (t-t_k)_+^\alpha,
\end{equation}
for some coefficients $a_k \in \R$ such that $\sum_k |a_k| < \infty$, and an ordered sequence of numbers $\cdots < t_{-1} < t_0 < t_1 < \cdots$ called knots. Such functions were introduced in \cite{Unser_Blu_2000} as generalisations of polynomial splines and have since found applications in various areas of science and engineering involving fractional-order signals \cite{Monje_Chen_Vinagre_Xue_Feliu_2010, Sheng_Chen_Qiu_2012}. A natural class of $\alpha$-\Holder continuous stochastic processes
\begin{equation}
 \label{eq:RandomSpline}
 X_t \weq \sum_{k \in \Z} \xi_k (t-t_k)_+^\alpha,
\end{equation}
is obtained by replacing the coefficients $a_k$ in \eqref{eq:Spline} by random variables $\xi_k$. Then a fundamental question in nonlinear stochastic control theory is to know when $t \mapsto f(X_t)$ is integrable against the paths of a $\beta$-\Holder continuous process $Y$.

If we assume that $\xi_k = w_k \eta_k$ for some constants $w_k$ such that $\sum_k \abs{w_k} < \infty$ and for some independent and identically distributed random variables $\eta_k$ having a bounded probability density function and finite moments of order $r > (\frac{\alpha}{1-\beta}-1)^{-1}$, then the process $X$ defined by \eqref{eq:RandomSpline} satisfies Assumption \ref{ass:Density} and the assumptions of Theorem \ref{the:RSIntegral}, so that $f \circ X$ is Riemann--Stieltjes integrable against an arbitrary $\beta$-\Holder continuous process, for a wide class of potentially discontinuous functions $f$.
\end{example}
\end{echangedi}

\section{Proofs}
\label{sec:proofs}
\subsection{Fractional integrals and derivatives}
\label{subsec:fractional_integrals}
\subsubsection{Riemann--Liouville operators}
Let us recall some basic facts and notations of fractional integrals and derivatives from \cite{Samko_Kilbas_Marichev_1993}. Let $(a,b) \subset \R$ be a nonempty bounded interval. For $p \in [1,\infty)$ (resp.\ $p=\infty$), we denote by $L_p = L_p(a,b)$ the space of $p$-integrable (resp.\ essentially bounded) real functions on $(a,b)$. The fractional left and right Riemann--Liouville integrals of order $\alpha > 0$ of a function $f \in L_1$ are denoted\footnote{Here we follow Zähle's (and Liouville's) convention of including $(-1)^{-\alpha} = e^{- i \pi \alpha}$ into the definition of right-sided fractional integrals \cite{Zahle_1998}.} by
\[
 I^\alpha_{a+}f(t)
 \weq \frac{1}{\Gamma(\alpha)} \int_a^t \frac{f(s)}{(t-s)^{1-\alpha}} \, ds
\]
and
\[
 I^\alpha_{b-}f(t)
 \weq \frac{(-1)^{-\alpha}}{\Gamma(\alpha)} \int_t^b \frac{f(s)}{(t-s)^{1-\alpha}} \, ds.
\]
The above integrals converge for almost every $t \in (a,b)$, and $I^\alpha_{a+} f$ and $I^\alpha_{b-}$ may be considered as functions in $L_1$. By convention, $I^0_{a+}$ and $I^0_{b-}$ are defined as the identity operator. Moreover, the integral operators $I^\alpha_{a+}, I^\alpha_{b-}: L_1 \to L_1$ are linear and one-to-one. The inverse operators are known as Riemann--Liouville fractional derivatives, and denoted by $I^{-\alpha}_{a+} = (I^\alpha_{a+})^{-1}$ and $I^{-\alpha}_{b-} = (I^\alpha_{b-})^{-1}$. Because $L_p(a,b) \subset L_1(a,b)$ for every $p \in [1,\infty]$, we may also consider the fractional integrals as linear operators from $L_p$ into $L_1$, with ranges denoted by $I^\alpha_{a+}(L_p)$ and $I^\alpha_{b-}(L_p)$.

For any $\alpha \in (0,1)$ and for any $f \in I^\alpha_{a+}(L_1)$ and $g \in I^\alpha_{b-}(L_1)$, the Weyl--Marchaud derivatives 
\[
 D_{a+}^\alpha f(t)
 \weq \frac{1}{\Gamma(1-\alpha)}\left( \frac{f(t)}{(t-a)^\alpha} + \alpha \int_a^t \frac{f(t)-f(s)}{(t-s)^{\alpha+1}} \, ds \right)
\]
and
\[
 D_{b-}^\alpha g(t)
 \weq \frac{(-1)^\alpha}{\Gamma(1-\alpha)}\left( \frac{g(t)}{(b-t)^\alpha} + \alpha \int_t^b \frac{g(t)-g(s)}{(s-t)^{\alpha+1}} \, ds \right)
\]
are well defined, and coincide with the Riemann--Liouville derivatives according to $D_{a+}^\alpha f(t) = I^{-\alpha}_{a+} f(t)$ and $D_{b-}^\alpha g(t) = I^{-\alpha}_{b-} g(t)$ for almost every $t \in (a,b)$.

\subsubsection{Fractional Zähle--Stieltjes integrals}

Let $f$ and $g$ be measurable real functions defined on some set containing the interval $(a,b)$. Assume that the limits $f(a+), g(a+), g(b-)$ exist in $\R$, and denote $f_{a+}(t) = f(t) - f(a+)$ and $g_{b-}(t) = g(t) - g(b-)$. When $f_{a+} \in I^\alpha_{a+}(L_p)$ and $g_{b-} \in I^{1-\alpha}_{b-}(L_q)$ for some $\alpha \in [0,1]$ and $p,q \in [1,\infty]$ such that $1/p+1/q =1$,
the fractional version of the Stieltjes integral introduced by Zähle \cite{Zahle_1998} is defined by
\begin{equation}
 \label{eq:ZSIntegral}
 \begin{aligned}
 \ZSint_a^b f_t \, dg_t
 &\weq (-1)^\alpha \int_a^b D^{\alpha}_{a+} (f-f(a+))(t) \, D^{1-\alpha}_{b-} (g-g(b-))(t) \, dt \\
 &\qquad \qquad + f(a+)(g(b-)-g(a+)),
 \end{aligned}
\end{equation}
where the right side does not depend on $\alpha$. 



\subsubsection{Fractional Sobolev--Slobodeckij spaces}


To help treating fractional integrals and derivatives in a convenient way, we introduce the following fractional Sobolev--Slobodeckij type space below. For $\alpha > 0$, we denote by $W^\alpha_1(a+)$ the space of measurable functions $f: (a,b) \to \R$ such that $f(a+) \in \R$ exists and
\[
 ||f||_{W^\alpha_1(a+)}
 \weq \int_a^b \frac{|f(t)|}{(t-a)^\alpha} \, dt + \int_a^b \int_a^t \frac{|f(t)-f(s)|}{|t-s|^{1+\alpha}} \, ds \, dt
\]
is finite. We denote by $W^\alpha_\infty(b-)$ the space of measurable functions $f: (a,b) \to \R$ such that $f(b-) \in \R$ exists and
\[
 ||f||_{W^\alpha_\infty(b-)}
 \weq \sup_{t \in (a,b)} \frac{|f(b-) - f(t)|}{(b-t)^\alpha} + \sup_{t \in (a,b)} \int_t^b \frac{|f(t)-f(s)|}{|t-s|^{1+\alpha}} \, ds 
\]
is finite. The spaces $W^\alpha_1(a+)$ and $W^\alpha_\infty(b-)$ contain all \Holder continuous functions of order $\alpha'>\alpha$.
More importantly, these spaces also contain discontinuous functions of infinite power variation, as we will see later.


The following result may be considered as a mild generalisation of a similar estimate in \cite{Nualart_Rascanu_2002}, where a subspace of \Holder continuous functions was employed in place of $W^\alpha_\infty(b-)$. See also \cite[Theorem 1.1]{Zahle_2001} for a similar result for square integrable functions. 
\begin{proposition}
\label{the:ZSIntegralBound}
Assume that $f \in W^\alpha_1(a+)$ and $g \in W^{1-\alpha}_\infty(b-)$ for some $\alpha \in (0,1)$. Then the integral in \eqref{eq:ZSIntegral} is well defined, representable as
\begin{equation}
 \label{eq:ZSIntegralSimple}
 \ZSint_a^b f_t \, dg_t
 \weq (-1)^\alpha \int_a^b D^{\alpha}_{a+} f(t) \, D^{1-\alpha}_{b-} (g-g(b-))(t) \, dt,
\end{equation}
and bounded by
\begin{equation}
 \label{eq:ZSIntegralBound}
 \left|\ZSint_a^b f_t \, dg_t \right|
 \wle \frac{||f||_{W^\alpha_1(a+)} \, ||g||_{W^{1-\alpha}_\infty(b-)}}{\Gamma(\alpha) \Gamma(1-\alpha)}.
\end{equation}
\end{proposition}
\begin{proof}
The definitions of $W^\alpha_1(a+)$ and $W^{1-\alpha}_\infty(b-)$ guarantee that the Weyl--Marchaud derivatives $D_{a+}^\alpha f$ and $D_{b-}^{1-\alpha} (g-g(b-))$ are well defined and satisfy
\[
 || D^\alpha_{a+} f ||_{L_1}
 \wle \frac{||f||_{W^\alpha_1(a+)}}{\Gamma(1-\alpha)}
\]
and
\[
 || D^{1-\alpha}_{b-}(g-g(b-)) ||_{L_\infty}
 \wle \frac{||g||_{W^{1-\alpha}_\infty(b-)}}{\Gamma(\alpha)}.
\]
With the help of \cite[Theorem 13.2]{Samko_Kilbas_Marichev_1993}, one can verify that in this case $f \in I^\alpha_{a+}(L_1)$ and $g-g(b-) \in I^{1-\alpha}_{b-}(L_\infty)$. Then \eqref{eq:ZSIntegralSimple} follows by \cite[Eq.~(22')]{Zahle_1998}, and \eqref{eq:ZSIntegralBound} follows by plugging in the above two estimates into \eqref{eq:ZSIntegralSimple}.
\end{proof}

\subsection{\Holder and Gagliardo seminorms}

The \Holder seminorm of order $\alpha > 0$ of a measurable function $x: [0,T] \to \R$ is denoted by
\[
 [x]_{\alpha,\infty}
 \weq \sup_{0 \le s < t \le T} \frac{|x(t) - x(s)|}{|t-s|^\alpha},
\]
and the Gagliardo seminorm of order $\alpha>0$ and exponent $p \in [1,\infty)$ by
\begin{equation}
 \label{eq:Gagliardo}
 [x]_{\alpha, p}
 \weq \left( \int_0^T \int_0^T \frac{|x(t) - x(s)|^p}{|t-s|^{1+\alpha p}} \, ds \, dt \right)^{1/p}.
\end{equation}

\begin{lemma}
\label{the:HolderGagliardo}
For any $\alpha > 0$ and $\epsilon>0$,
\begin{align*}
 ||x||_{W^\alpha_1(0+)}
 &\wle (1-\alpha)^{-1} T^{1-\alpha} ||x||_\infty + \frac{1}{2} [x]_{\alpha,1}, \\
 ||x||_{W^\alpha_\infty(T-)}
 &\wle (1+\epsilon^{-1}) T^\epsilon [x]_{\alpha+\epsilon,\infty},
\end{align*}
and
\[
 [x]_{\alpha,1}
 \wle \epsilon^{-1}(1+\epsilon)^{-1} T^{1+\epsilon}  [x]_{\alpha+\epsilon,\infty}.
\]
\end{lemma}
\begin{proof}
These are straightforward computations.
\end{proof}

\subsection{Regularity of composite paths}

In this section we analyse the regularity of $t \mapsto f \circ X(t)$ when $f$ is a function of locally finite variation and $t \mapsto X(t)$ is a random or nonrandom \Holder continuous function. The main result is the following.

\begin{proposition}
\label{the:GagliardoComposite}
Let $X$ be an $\alpha$-\Holder continuous random process such that Assumption~\ref{ass:Density} holds. Then for any function $f$ of locally finite variation, $[f \circ X]_{\theta,p} < \infty$ almost surely for all $\theta \in (0,1)$ and $p \in [1,\infty)$ such that $1 \le p < \alpha/\theta$.
\end{proposition}

Any function $f: \R \to \R$ of locally finite variation can be written as a sum of a right-continuous and a left-continuous function of locally finite variation \cite[Prop.~2.19]{Kallenberg_2002}. This is why in the sequel we will mainly discuss right-continuous functions. If we ignore the value of $f$ at its jump points, then we may define the right-continuous function $f_r(x) = \lim_{h \downarrow 0} f(x+h)$, and take this function as a representative of $f$.

\begin{lemma}
\label{the:RightContinuity}
Assume that $\Leb\{t \in [0,T]: x_t \in J_f\} = 0$ where $J_f$ denotes the set of discontinuities of $f$. Then for any $\theta \in (0,1)$ and $p \in [1,\infty)$,\[
 [f \circ x]_{\theta,p}
 \weq  [f_r \circ x]_{\theta,p}.
\]
\end{lemma}
\begin{proof}
The claim follows by noting that 
\[
 \frac{|f(X_t) - f(X_s)|^p}{|t-s|^{1+\theta p}}
 \weq \frac{|f_r(X_t) - f_r(X_s)|^p}{|t-s|^{1+\theta p}}
\]
for all $(s,t)$ outside the set $\{(s,t): x_s \in J_f \ \text{or} \ x_t \in J_f\}$ which has $\Leb^2$-measure zero by assumption.
\end{proof}

Note that because $J_f$ is countable for any $f$ of locally finite variation, another sufficient condition for the above result is to assume that $\Leb\{t: x_t = y\} = 0$ for all $y$.

\begin{lemma}
\label{the:HolderSingular}
If $\theta > 0$ and $x: [0,T] \to \R$ is \Holder continuous of order $\alpha >0$, then
\[
 \int_0^t \frac{\ind{x_s < y}}{(t-s)^{1+\theta}} \, ds
 \wle \theta^{-1} [x]_{\alpha,\infty}^{\theta/\alpha } \,  |x_t-y|^{-\theta/\alpha }
\]
for all $t \in [0,T]$ and for all $y < x_t$.
\end{lemma}
\begin{proof}
Fix $t \in [0,T]$ and $y < x_t$, and denote by $I$ the integral above. If $x_s \ge y$ for all $s \in [0,t]$, then $I=0$ and there is nothing to prove. Let us consider the case where $x_s < y$ for at least one $s \in [0,t]$. Let $s^* = \sup\{s \in [0,t]: X_s < y\}$ be the supremum of the time indices for which this happens. Then $s^* \in [0,t)$ and $x_{s^*} = y$ by the continuity of $x$. In this case the integral is bounded from above by
\begin{equation}
 \label{eq:Integral2}
 I
 \wle \int_0^{s^*} \frac{ 1 }{(t-s)^{1+\theta}} \, ds
 \weq \frac{(t-s^*)^{-\theta } - t^{-\theta }}{\theta }
 \wle \theta ^{-1} (t-s^*)^{-\theta }.
\end{equation}
By the \Holder continuity of $x$, it follows that
\[
 | x_t - y |
 \weq | x_t - x_{s^*} |
 \wle [x]_{\alpha ,\infty} \, |t-s^*|^{\alpha },
\]
so that
\[
 I
 \wle \theta ^{-1} (t-s^*)^{-\theta }
 \wle \theta ^{-1} [x]_{\alpha ,\infty}^{\theta /\alpha } \, |x_t-y|^{-\theta /\alpha }.
\]
\end{proof}

For a right-continuous function $f$ of locally finite variation, denote its variation measure by $\mu_f$. This is the unique locally finite Borel measure on $\R$ such that the variation of $f$ over the interval $[a,b]$ equals $V_{(a,b)}(f) = \mu_f(a,b]$ for all $a < b$.
\begin{proposition}
\label{the:GagliardoNonrandom}
Let $f: \R \to \R$ be right-continuous and of finite variation. Let $x: [0,T] \to \R$ be \Holder continuous of order $\alpha > 0$.
\changedi{Then the Gagliardo seminorm of order $\theta  \in (0,1)$ and exponent $p \in [1,\infty)$ of the composite path $f \circ x$ is bounded by}
\begin{equation}
 \label{eq:GagliardoNonrandom}
  [f \circ x]_{\theta ,p}^p
  \wle \changedi{2^{p+1}} (\theta  p)^{-1} \mu_f(K_x)^{p-1} [x]_{\alpha ,\infty}^{\theta  p/\alpha } \, \int_0^T  \int_{K_x} |x_t - y|^{-\theta  p/\alpha } \, \mu_f(dy) \, dt,
\end{equation}
where $K_x$ is the closure of the range of $x$.
\end{proposition}
\begin{proof}
(i) Let us first assume that $f$ is increasing and right-continuous. In this case the variation measure $\mu_f$ equals the Lebesgue--Stieltjes measure of $f$ and satisfies $\mu(x_1,x_2] = f(x_2)-f(x_1)$ for $x_1 < x_2$. Let $G_x = \{(t,x_t): t \in [0,T]\}$ be the graph of $x$. We may and will assume that $(\Leb \otimes \mu_f)(G_x) = 0$ because otherwise the right side of~\eqref{eq:GagliardoNonrandom} is infinite and there is nothing to prove.

Observe next that for all real numbers $x_1$ and $x_2$ in $K=K_x$,
\[
 f(x_2)-f(x_1)
 \weq
 \begin{cases}
   \int_{K} 1_{(x_1,x_2]}(y)   \, \mu_f(dy), &\quad x_1 \le x_2,\\
   - \int_{K} 1_{(x_2,x_1]}(y) \, \mu_f(dy), &\quad x_1 > x_2,
 \end{cases}
\]
so that
\[
 |f(x_t)-f(x_s)|
 \weq \int_{K} 1_{(x_s,x_t]}(y) \, \mu_f(dy) + \int_{K} 1_{(x_t,x_s]}(y) \, \mu_f(dy).
\]
\changedi{Because $|f(x_t)-f(x_s)|^{p-1} \le \mu(K)^{p-1}$, it hence follows that}
\[
 \abs{f(x_t)-f(x_s)}^p
 \wle \mu(K)^{p-1} \left\{ \int_{K} 1_{(x_s,x_t]}(y) \, \mu_f(dy) + \int_{K} 1_{(x_t,x_s]}(y) \, \mu_f(dy) \right\},
\]
and the Gagliardo seminorm of $f \circ X$ can be bounded by
\[
 [f \circ X]_{\theta,p}^p
 \wle 2 \mu(K)^{p-1} \int_0^T \int_0^t \int_{K} \frac{1_{(x_s,x_t]}(y) + 1_{(x_t,x_s]}(y)}{(t-s)^{1+\theta  p}} \, ds dt \mu_f(dy).
\]
Because $\Leb \otimes \mu_f(G_x) = 0$, it follows that the set
\[
 \tilde A
 \weq \left\{(s,t,y) \in [0,T]^2 \times K: \ x_s = y \ \text{or} \ x_t = y \right\}
\]
has $\Leb^2 \otimes \mu_f$ -measure zero. Because 
\[
 \frac{1_{(x_s,x_t]}(y) + 1_{(x_t,x_s]}(y)}{(t-s)^{1+\theta  p}}
 \weq  \frac{1_{(x_s,x_t)}(y) + 1_{(x_t,x_s)}(y)}{(t-s)^{1+\theta  p}}
\]
for all $(s,t,y)$ outside the set $\tilde A$, we conclude that 
\begin{equation}
 \label{eq:GagliardoSplit}
 [f \circ X]_{\theta,p}^p
 \wle 2 \mu(K)^{p-1} \left\{ \int_{K} J_1(y) \mu_f(dy) + \int_{K} J_2(y) \mu_f(dy) \right\},
\end{equation}
where
\[
 J_1(y) \weq \int_0^T \int_0^t \frac{1_{(x_s,x_t)}(y)}{(t-s)^{1+\theta  p}} \, ds dt
 \quad\text{and}\quad
 J_2(y) \weq \int_0^T \int_0^t \frac{1_{(x_t,x_s)}(y)}{(t-s)^{1+\theta  p}} \, ds dt.
\]

To estimate the above integrals, let us first rewrite $J_1(y)$ as an integral
$
 J_1(y) \weq \int_0^T J_{1,t}(y) \, 1(x_t > y) \, dt,
$
where
\[
 J_{1,t}(y)
 \weq \int_0^t \frac{\ind{x_s < y}}{(t-s)^{1+\theta  p}} \, ds.
\]
By Lemma~\ref{the:HolderSingular},
\[
 J_{1,t}(y)
 \wle (\theta  p)^{-1} [x]_{\alpha ,\infty}^{\theta  p/\alpha } |x_t-y|^{-\theta  p/\alpha }
\]
for all $t$ such that $x_t > y$. Hence
\[
  J_1(y)
  \wle (\theta  p)^{-1} [x]_{\alpha ,\infty}^{\theta  p/\alpha } \, \int_0^T  |x_t - y|^{-\theta  p/\alpha } \, dt.
\]

By applying the same argument to the function $t \mapsto -x_t$ at the point $-y$, we may also conclude that
\[
 J_{2,t}(y)
 \, := \ \int_0^t \frac{\ind{x_s > y} }{(t-s)^{1+\theta  p}} \, ds
 \wle (\theta  p)^{-1} [x]_{\alpha ,\infty}^{\theta  p/\alpha } |x_t-y|^{-\theta  p/\alpha }
\]
for all $t$ such that $x_t < y$. Therefore
\[
  J_2(y)
  \weq \int_0^T J_{2,t}(y) 1(x_t < y) \, dt
  \wle (\theta  p)^{-1} [x]_{\alpha ,\infty}^{\theta  p/\alpha } \, \int_0^T  |x_t - y|^{-\theta  p/\alpha } \, dt.
\]
By combining these inequalities with~\eqref{eq:GagliardoSplit}, we obtain~\eqref{eq:GagliardoNonrandom} \changedi{with the constant $2^{p+1}$ on the right side replaced by 4}.

(ii) Let us now consider the general case with a right-continuous function $f$ of finite variation. Then we can represent $f$ as
\[
 f(t) = f(0) + f^+(t) - f^-(t),
\]
where $f^+$ and and $f^-$ are the positive and negative variation functions of $f$ defined by
\[
 f^\pm(t) \weq
 \begin{cases}
   V^\pm(0,t), &\quad t \ge 0, \\
   -V^\pm(t,0), &\quad t < 0.
 \end{cases}
\]
The functions $f^+$ and $f^-$ are increasing, bounded, and right-continuous \cite[Prop.~2.19]{Kallenberg_2002}, and the variation measure of $f$ can be represented as $\mu_f = \mu_{f^+} + \mu_{f^-}$.
Because
\[
 [f \circ x]_{\theta ,p}^p
 \wle 2^{p-1} \left( [f^+ \circ x]_{\theta,p}^p + [f^- \circ x]_{\theta ,p}^p \right),
\]
the claim follows by part (i).
\end{proof}

\begin{proof}[Proof of Proposition~\ref{the:GagliardoComposite}]
Because $f$ is of locally finite variation, its set of discontinuities $J_f$ is countable. Because $X_t$ admits a probability density function for almost every $t \in [0,T]$, it follows that
\[
 \E \Leb\{t:X_t \in J_f\}
 \weq \E \int_0^T 1_{J_f}(X_t) \, dt
 \weq \int_0^T \pr(X_t \in J_f) \, dt
 \weq 0,
\]
so that the set $\{t \in [0,T]: X_t \in J_f\}$ has Lebesgue measure zero with probability one. Hence by Lemma~\ref{the:RightContinuity}, $[f \circ X]_{\theta ,p} = [f_r \circ X]_{\theta ,p}$ almost surely, where $f_r$ denotes the right-continuous \changedi{modification} of $f$. Hence in what follows, we may and will assume that $f$ is right-continuous.

Proposition~\ref{the:GagliardoNonrandom} implies that $[f \circ X]_{\theta,p}^p \le \changedi{2^{p+1}} (\theta  p)^{-1} [X]_{\alpha ,\infty}^{\theta  p/\alpha } \, Q$, where
\[
 Q
 \weq \mu_f(K_X)^{p-1} \int_0^T  \int_{K_X} |X_t - y|^{-\theta  p/\alpha } \, \mu_f(dy) \, dt,
\]
and $K_X$ is the closure of the range of $X$.
Because the path of $X$ is $\alpha $-\Holder continuous almost surely, it suffices to verify that $Q < \infty$ almost surely. This is what we shall do next.

On the event $A_k = \{||X||_\infty \le k\}$, the random variable $Q$ is bounded from above by
\[
 Q_k
 \weq \mu_f([-k,k])^{p-1} \int_0^T  \int_{\left[-k, k \right]} |X_t - y|^{-\theta  p/\alpha } \, \mu_f(dy) \, dt.
\]
\changedi{
Now by using Fubini's theorem, we find that
\begin{align*}
 \E Q_k
 &\weq \mu_f([-k,k])^{p-1} \int_{\left[-k, k \right]} \int_0^T \E |X_t - y|^{-\theta  p/\alpha } \, dt \, \mu_f(dy) \\
 &\wle \mu_f([-k,k])^p \, \sup_{y \in \R} \int_0^T \E |X_t - y|^{-\theta  p/\alpha } \, dt,
\end{align*}
where $\mu_f([-k,k]) = V_f(-k,k)$ is the variation of $f$ on $[-k,k]$.}
We know (Lemma~\ref{the:Tikanmaki}) that
\[
 \E |X_t - y|^{-\theta  p/\alpha }
 \wle 1 + \frac{2}{1-\theta  p/\alpha } \sup_x p_t(x)
 \wle 1 + \frac{2}{1-\theta  p/\alpha } \hat p_t
\]
for all $y \in \R$ and for almost every $t$. By integrating the left and the right side of the above inequality over $t$, we find that
\[
 \E Q_k
 \wle V_f(-k,k)^p \left( T + \frac{2}{1-\theta p /\alpha } \int_0^T \hat p_t \, dt \right).
\]
Because $f$ has locally finite variation, it follows that $\E Q_k < \infty$, and hence $Q_k$ is finite almost surely.

We may now conclude that for all $k \ge 1$,
\[
 \pr(Q=\infty, A_k)
 \wle \pr(Q_k = \infty, A_k)
 \wle \pr(Q_k = \infty)
 \weq 0.
\]
The almost sure continuity of $X$ implies that $||X||_\infty$ is finite almost surely, and therefore $\pr(\cup_k A_k) = 1$. Hence
\[
 \pr(Q=\infty)
 \weq \pr(Q=\infty, \cup_k A_k)
 \wle \sum_k \pr(Q=\infty, A_k)
 \weq 0.
\]
\end{proof}

\subsection{Proof of Theorem~\ref{the:ZSIntegral}}

\begin{proof}[Proof of Theorem~\ref{the:ZSIntegral}]
Choose some $\theta  \in (1-\beta,\alpha)$. By Proposition~\ref{the:ZSIntegralBound}, it suffices to show that $f \circ X \in W^\theta _1(0+)$ and 
$Y \in W^{1-\theta }_\infty(T-)$ with probability one. Observe first that $f \circ X(0+)$ exists almost surely by assumption. Moreover, by Lemma~\ref{the:HolderGagliardo}, 
\[
 ||f \circ X||_{W^\theta _1(0+)}
 \wle c (||f \circ X||_\infty + [f \circ X]_{\theta ,1})
\]
for $c = (1-\theta )^{-1} T^{1-\theta } + \frac{1}{2}$. The term $||f \circ X||_\infty$ is finite almost surely, because $X$ is continuous and $f$ is locally bounded. The Gagliardo seminorm $[f \circ X]_{\theta ,1}$ is almost surely finite by Proposition~\ref{the:GagliardoComposite}. Hence $f \circ X \in W^\theta _1(0+)$ almost surely.

Because $Y$ is almost surely \Holder continuous of order $\beta > 1-\theta $, we see that $||Y||_{W^{1-\theta }_\infty(T-)}$ is finite by Lemma~\ref{the:HolderGagliardo}. The \Holder continuity of $Y$ also implies that $Y(T-)$ exists almost surely. Therefore $Y \in W^{1-\theta }_\infty(T-)$ with probability one, and the claim now follows by Proposition~\ref{the:ZSIntegralBound}.
\end{proof}

\subsection{Proof of Theorem~\ref{the:Ito}}

Assume without loss of generality that $t=T$.
Denote by $\phi = f'$ the derivative of $f$. We define a smooth approximation $\phi_n$ of $\phi$ as in \eqref{eq:Mollification}. Define also a smooth approximation $f_n$ of $f$ by the same formula. Then $f_n$ is infinitely differentiable with derivative $f_n' = \phi_n$.
Because $\phi_n$ is Lipschitz continuous, we see that $\phi_n \circ X$ is \Holder continuous of order $\alpha > 1/2$, and it follows \cite[Theorem 4.3.1]{Zahle_1998} that
\begin{equation}
 \label{eq:ItoSmooth}
 f_n(X_T) - f_n(X_0)
 \weq \ZSint_0^T \phi_n(X_t) \, dX_t	
\end{equation}
almost surely.

We will next verify that \eqref{eq:ItoSmooth} remains valid in the limit as $n \to \infty$.
By Proposition~\ref{the:SmoothGagliardo}, $[\phi_n \circ X - \phi \circ X]_{\theta,1} \prto 0$.
Observe next that the definition of $\phi_n$ guarantees that for any $c>0$,
\[
 \sup_{|x|\le c} |\phi_n(x) - \phi(x)|
 \wle \sup_{|x|\le c+1} |\phi(x)| + \sup_{|x| \le c} |\phi(x)|
\]
for all $n$. Therefore,
\[
  t^{-\theta} |\phi_n(X_t) - \phi(X_t)|
 \wle 2 \, t^{-\theta} \sup_{|x| \le ||X||_\infty + 1} |\phi(x)|,
\]
where $||X||_\infty = \sup_{t \in [0,T]} |X_t|$ is almost surely finite. Because $\phi_n \to \phi$ pointwise and the right side above is almost surely integrable with respect to the Lebesgue measure on $[0,T]$, we conclude by dominated convergence that
\[
 \int_0^T \frac{|\phi_n(X_t) - \phi(X_t)|}{t^\theta} \, dt
 \ \to \ 0
\]
almost surely. Our conclusion is that $||\phi_n \circ X - \phi \circ X||_{W^\theta_1(0+)} \to 0$ almost surely. Hence by Proposition~\ref{the:ZSIntegralBound}, the right side of~\eqref{eq:ItoSmooth} converges according to
\[
 \ZSint_0^T \phi_n(X_t) \, dX_t
 \ \prto \ \ZSint_0^T \phi(X_t) \, dX_t.
\]
On the other hand, the left side of~\eqref{eq:ItoSmooth} converges almost surely to $f(X_T) - f(X_0)$. Because the limits must agree with probability one, we conclude that
\[
 f(X_T) - f(X_0)
 \weq \ZSint_0^T \phi(X_t) \, dX_t	
\]
almost surely, and the proof of Theorem~\ref{the:Ito} is complete.

\subsection{Riemann--Stieltjes integrability}

In this section we will prove Theorem~\ref{the:RSIntegral} and Theorem~\ref{the:RateMean}. Let $X$ and $Y$ be \Holder continuous random processes on $[0,T]$ and $f$ a function of locally finite variation fulfilling the assumptions of  Theorem~\ref{the:RSIntegral}. Fix a tagged partition $\pi = (t_0,t_1,\dots, t_k; \xi_1,\dots,\xi_k)$ such that $0=t_0 < t_1 < \cdots < t_k = T$ and $\xi_i \in [t_{i-1},t_i]$ for all $i$. Because $Y$ is continuous on $[0,T]$, we see that (cf.\ \cite[Proposition 2.2]{Zahle_1998})
\[
 Y_{t_i} - Y_{t_{i-1}}
 \weq \ZSint_0^T 1_{(t_{i-1},t_i]}(t) \, dY_t
\]
for all $i$. Hence the Riemann--Stieltjes sum of $f \circ X$ against $Y$ along the tagged partition $\pi$ can be written as
\[
 \RSsum(f\circ X, Y, \pi)
 \weq \ZSint_0^T \sum_{i=1}^k f(X_{\xi_i}) 1_{(t_{i-1},t_i]}(t) \, dY_t
 \weq \ZSint_0^T f(\hat X_\pi(t)) \, dY_t,
\]
where a $\pi$-interpolated version of $X$ is defined by
\[
 \hat X_\pi(t)
 \weq \sum_{i=1}^k X_{\xi_i} 1_{(t_{i-1},t_i]}(t).
\]
By Proposition~\ref{the:ZSIntegralBound}, we find that the approximation error
\begin{equation}
 \label{eq:Delta}
 \Delta(\pi)
 \weq \RSsum(f\circ X, Y, \pi) - \ZSint_0^T f(X(t)) \, dY(t) 
\end{equation}
is bounded by
\begin{equation}
 \label{eq:DeltaBound}
\begin{aligned}
 \left| \Delta(\pi) \right|
 \weq \left| \ZSint_0^T (f(\hat X_\pi(t)) - f(X_t)) \, dY_t \right|
 \wle \frac{|| h_\pi ||_{W^\theta_1(0+)} || Y ||_{W^{1-\theta}_\infty(T-)}}{\Gamma(\theta) \Gamma(1-\theta)},
\end{aligned}
\end{equation}
where
\[
 h_\pi(t) 
 \weq f \circ \hat X_\pi(t) - f \circ X(t)
\]
denotes the interpolation error at $t$ \changedi{(see below for the choice of $\theta$).}

\begin{proof}[Proof of Theorem~\ref{the:RSIntegral}]
Suppose that the assumptions of Theorem \ref{the:RSIntegral} prevail with \changedii{$0 < 1-\beta < H-1/r$, and assume that $0 < a < \frac{r}{1+r}H - (1-\beta)$. Choose a number $\theta \in (0,1)$ so that $1-\beta < \theta < H-1/r$ and $\theta < \frac{r}{1+r}H - a$.
}
Observe that by Lemma~\ref{the:HolderGagliardo},
\[
 ||Y||_{W^{1-\theta}_\infty(T-)}
 \wle c_1 [Y]_{\beta,\infty},
\]
where $c_1=(1+\epsilon_1^{-1})T^{\epsilon_1}$ with $\epsilon_1 = \beta - (1-\theta)$, so that by \eqref{eq:DeltaBound}, the error term in \eqref{eq:Delta} is bounded by
\[
 |\Delta(\pi)|
 \wle c_2 ||h_\pi||_{W^\theta_1(0+)} [Y]_{\beta,\infty}
\]
for $c_2 = c_1 \Gamma(\theta)^{-1} \Gamma(1-\theta)^{-1}$.

(i) Assume first that $f$ is right-continuous and has finite variation. Now by Lemma~\ref{the:Key}, $\E || h_\pi ||_{W^\theta_1(0+)} \le c d_1(f) ||\pi||^\rho$ for $\rho = \frac{r}{1+r}H - \theta > 0$. Hence $|| h_\pi ||_{W^\theta_1(0+)} \prto 0$ as $||\pi|| \to 0$. Because $[Y]_{\beta,\infty}$ is finite almost surely, we conclude that $|\Delta(\pi)| \prto 0$ as $||\pi|| \to 0$.

(ii) Assume next that $f$ is right-continuous with locally, but not globally, finite variation. Let $A_k = \{||X||_\infty \le k\}$. On this event we may replace $f$ by its truncated version $f^{(k)}$ which has finite variation. On this event $h_\pi = h^{(k)}_\pi$ where the latter random function corresponds to replacing $f$ by $f^{(k)}$. By the first part, $|| h^{(k)}_\pi ||_{W^\theta_1(0+)} \prto 0$. Because $h^{(k)}_\pi = h_\pi$ on the event $A_k$, we conclude that $1_{A_k} || h^{(k)}_\pi ||_{W^\theta_1(0+)} \prto 0$ for all $k \ge 1$. Because $\pr(\cup_k A_k) = \pr( ||X||_\infty < \infty) = 1$, we conclude (Lemma~\ref{the:ConvergenceInProbability}) that $|| h_\pi ||_{W^\theta_1(0+)} \prto 0$ also in this case.

(iii) Assume now that
\[
 \sum_{n=1}^\infty ||\pi_n||^a < \infty.
\]
\removedii{for some $a < \frac{\lambda}{1+\delta} + \beta - 1$. We may and will assume that this $a$ is the same $a$ that was fixed in the beginning of the proof.} Then for any right-continuous function of finite variation, by applying Lemma~\ref{the:Key} we see that
\[
 \E \bigg( \sum_{n \ge 1} ||h_{\pi_n}||_{W^\theta_1(0+)} \bigg)
 \wle \sum_{n \ge 1} c ||\pi_n||^{\changedii{\frac{r}{1+r}H - \theta} }
 \wle c \sum_{n \ge 1} ||\pi_n||^a
 < \infty.
\]
Hence $||h_{\pi_n}||_{W^\theta_1(0+)} \to 0$ almost surely, and \changedi{hence also $\abs{\Delta(\pi_n)} \to 0$} almost surely. Making the same conclusion for any right-continuous $f$ of locally finite variation can be done as in (ii).
\end{proof}

\changedii{
\begin{proof}[Proof of Theorem~\ref{the:RateDyadic}]
Suppose that $0 < 1-\beta < H-1/r$ and $\epsilon > 0$, and denote $a = \frac{r}{1+r} H - (1-\beta) - \epsilon$. Choose a number $\theta \in (0,1)$ so that $1-\beta < \theta < H-1/r$ and $\theta < \frac{r}{1+r}H - a$. By using the same notations and reasoning as in the proof of Theorem~\ref{the:RSIntegral}, we find that the absolute value of the error term $\Delta(\pi_n)$ is almost surely bounded by a constant multiple of $||h_{\pi_n}||_{W^\theta_1(0+)} [Y]_{\beta,\infty}$, and
\begin{align*}
 \E \bigg( \sum_{n \ge 1} 2^{a n} ||h_{\pi_n}||_{W^\theta_1(0+)} \bigg)
 &\wle c \sum_{n \ge 1} 2^{a n}  ||\pi_n||^{\frac{r}{1+r}H - \theta} \\
 &\weq c T^{\frac{r}{1+r}H - \theta} \sum_{n \ge 1} 2^{-(\frac{r}{1+r}H - \theta - a)n}
 < \infty.
\end{align*}
Hence the sequence $n \mapsto 2^{a n} ||h_{\pi_n}||_{W^\theta_1(0+)}$ is bounded almost surely by a finite (random) constant, and the conclusion follows.
\end{proof}
}

\begin{proof}[Proof of Theorem~\ref{the:RateMean}]
Suppose that the assumptions of Theorem \ref{the:RSIntegral} prevail with \changedii{$0 < 1-\beta < H-1/r$, and assume that $0 < a < \frac{r}{1+r}H - (1-\beta)$. Choose a number $\theta \in (0,1)$ so that $1-\beta < \theta < H-1/r$ and $\theta < \frac{r}{1+r}H - a$.
}
Let us next choose a small enough $p \in (1,2)$ so that
\[
 p \wle \changedii{\min\left( r, \ 1+\frac{r}{1+r}H, \ \frac{\frac{r}{1+r}H - \theta + 1}{a+1} \right)}.
\]
Observe that by Lemma~\ref{the:HolderGagliardo},
\[
 ||Y||_{W^{1-\theta}_\infty(T-)}
 \wle c_1 [Y]_{\beta,\infty},
\]
where $c_1=(1+\epsilon_1^{-1})T^{\epsilon_1}$ with $\epsilon_1 = \beta - (1-\theta)$, so that by \eqref{eq:DeltaBound}, the Riemann--Stieltjes integration error is bounded by
\[
 |\Delta(\pi)|
 \wle c_2 ||h_\pi||_{W^\theta_1(0+)} [Y]_{\beta,\infty}
\]
for $c_2 = c_1 \Gamma(\theta)^{-1} \Gamma(1-\theta)^{-1}$. If $f$ has finite total variation, then we find by Lemma~\ref{the:Key} and our choice of $p$ that (for $||\pi|| \le 1$)
\begin{equation}
 \label{eq:LpBound}
 \left( \E ||h_\pi||^p_{W^\theta_1(0+)} \right)^{1/p}
 \wle c_3 d_1(f)^{1/p} ||\pi||^{\changedii{(1+\frac{r}{1+r}H - \theta - p)/p}}
 \wle c_3 d_1(f)^{1/p} ||\pi||^a,
\end{equation}
so that for $1/q = 1-1/p$,
\[
 \E |\Delta(\pi)|
 \wle c_2 \left( \E ||h_\pi||^p_{W^\theta_1(0+)} \right)^{1/p} \left(\E [Y]^q_{\beta,\infty} \right)^{1/q}
 \wle c_2 c_3 d_1(f)^{1/p} \left(\E [Y]^q_{\beta,\infty} \right)^{1/q} ||\pi||^a,
\]
where
\[
 d_1(f) \weq
 \begin{cases}
 V(f)^p, &\quad \epsilon = 0, \\
 V(f)^p + \big(\sup_{|x| < \epsilon} |f'(x)| \big)^p, &\quad \epsilon > 0.
 \end{cases}
\]
This proves the claim under the assumption that $f$ has finite total variation.

Assume next that the total variation of $f$ is infinite. Denote by $A_k$ the event that $k-1 < ||X||_\infty \le k$. On this event, we see that
\[
 h_\pi
 \weq f \circ \hat X_\pi  - f \circ X
 \weq f^{(k)} \circ \hat X_\pi - f^{(k)} \circ X
 \ =: \ h^{(k)}_\pi,
\]
where $f^{(k)}$ is the truncated version of $f$ defined by
\[
 f^{(k)}(x)
 \weq 
 \begin{cases}
   f(-k), &\quad x < -k, \\
   f(x), &\quad |x| \le k, \\
   f(k), &\quad x > k.
 \end{cases}
\]
By the same argument as above,
\[
 |\Delta(\pi)| 1_{A_k}
 \wle c_2 ||h^{(k)}_\pi||_{W^\theta_1(0+)} [Y]_{\beta,\infty} 1_{A_k}.
\]
Then we may choose a large enough $q \in (1,\infty)$ so that $1/p + 2/q = 1$. Then by \Holder's inequality for three random variables,
\[
 \E ( |\Delta(\pi)| 1_{A_k} )
 \wle c_2 \left( \E ||h_\pi^{(k)}||_{W^{\theta}_1(0+)}^{p} \right)^{1/p} \left( \E ||Y||_{\beta,\infty}^{q} \right)^{1/q} \pr(A_k)^{1/q}.
\]
By applying Lemma~\ref{the:Key} again, we see that \eqref{eq:LpBound} is valid with $f$ replaced by $f^{(k)}$.
Therefore,
\[
 \E |\Delta(\pi)|
 \weq \E \sum_{k = 0}^\infty |\Delta(\pi)| 1_{A_k}
 \wle c_2 c_3 \left( \E ||Y||_{\beta,\infty}^{q} \right)^{1/q} \left( \sum_{k = 1}^\infty d_1(f^{(k)})^{1/p} \pr(A_k)^{1/q} \right) ||\pi||^a.
\]

Now because $f^{(k)}(x) = f(x)$ for $|x| \le k$,
\[
 d_1(f^{(k)}) \weq
 \begin{cases}
 V(f^{(k)})^p, &\quad \epsilon = 0, \\
 V(f^{(k)})^p + \big(\sup_{|x| < \epsilon} |f'(x)| \big)^p, &\quad \epsilon > 0.
 \end{cases}
\]
Because $f$ has infinite total variation, $V(f^{(k)}) = V_{[-k,k]}(f) \to \infty$, and it follows that
\[
 d_1(f^{(k)})^{1/p} \wle 2^{1/p} V(f^{(k)})
 \weq 2^{1/p} V_{[-k,k]}(f)
\]
for all large values of $k$. Because $\pr(A_k) \le \pr( ||X||_\infty > k-1)$, we find that
\[
 \left( \sum_{k = 1}^\infty d_1(f^{(k)})^{1/p} \pr(A_k)^{1/q} \right) < \infty,
\]
which confirms the validity of the claim.
\end{proof}

\begin{lemma}
\label{the:Key}
Let $f$ be a right-continuous function of finite variation. Let $X: [0,T] \to \R$ be a random process such that $X_t$ admits a probability density function $p_t$ \changedi{at almost every $t \in (0,T)$}, which is bounded according to Assumption~\ref{ass:Density}, and that
\[
 \changedi{c_0}
 \weq \sup_{0 \le s < t \le T} \changedii{\frac{\left(\E \abs{X_t - X_s}^{r}\right)^{1/r}}{|t-s|^{H}}} < \infty
\]
for some \changedii{$H \in (0,1]$ and $r > 1/H$}.
Assume also that there exist $\epsilon \ge 0$ and $\changedi{c_1} < \infty$ such that
\begin{itemize}
\item $p_t(x) \le c_1$ for all $|x| \ge \epsilon/2$ and \changedi{almost} all $t \in (0,T)$, and
\item $f$ restricted to $(-\epsilon,\epsilon)$ is Lipschitz continuous.
\end{itemize}
Then for any $1 \le p \le \changedii{r}$ and $0 < \theta < \min(2-p, \, \changedii{1+\frac{r}{1+r}H - p})$, there exists a finite constant $c$ such that the interpolation error $h_\pi = f \circ \hat X_\pi - f \circ X$ satisfies
\[
 \E  || h_\pi ||^p_{W^\theta_1(0+)}
 \wle c d_1(f) ||\pi||^{\changedii{1+\frac{r}{1+r}H - p - \theta}} 
\]
for any tagged partition $\pi$ \changedi{of mesh size $||\pi|| \le c_0^{\changedii{-1/H}}$}, where
\[
 d_1(f) \weq
 \begin{cases}
 V(f)^p, &\quad \epsilon = 0, \\
 V(f)^p + \big(\sup_{|x| < \epsilon} |f'(x)| \big)^p, &\quad \epsilon > 0.
 \end{cases}
\]
\end{lemma}
\begin{proof}
\changedi{By applying Proposition~\ref{the:WeakContinuity} with $q=\changedii{r}$, we see that there exists a constant $c_2$ such that
\begin{equation}
\label{eq:Avikainen2}
 \E |f(X_t)-f(X_s)|^p
 \wle c_2 d_1(f) |t-s|^{\changedii{v}},
\end{equation}
for almost every $t \in [0,T]$ and for every $s \in [0,T]$ such that $|t-s| \le c_0^{\changedii{-1/H}}$, where $v = \frac{r}{1+r}H$.}

Now consider a tagged partition $\pi$ of mesh size $||\pi|| \le c_0^{\changedii{-1/H}}$, and denote the intervals of the partition by $A_i = (t_{i-1},t_i]$. 
By applying Jensen's inequality we see that
\begin{equation}
 \label{eq:InterpolationError1}
\begin{aligned}
 ||h_\pi||_{W^\theta_1(0+)}^p
 &\weq \left( \int_0^T |h_\pi(t)| t^{-\theta} dt + \frac{1}{2} \int_0^T \int_0^T \frac{|h_\pi(t)-h_\pi(s)|}{|t-s|^{1+\theta}} ds dt \right)^p \\
 &\wle c_3 \int_0^T |h_\pi(t)|^p t^{-\theta} dt + c_4 \int_0^T \int_0^T \frac{|h_\pi(t)-h_\pi(s)|^p}{|t-s|^{p+\theta}} ds dt.
\end{aligned}
\end{equation}
where $c_3 = 2^{p-1} (\int_0^T t^{-\theta} dt)^{p-1}$ and $c_4 = 2^{-1} (\int_0^T \int_0^T |t-s|^{-\theta} ds dt)^{p-1}$.

To estimate the first term on the right side of \eqref{eq:InterpolationError1}, observe first that
\[
 |h_\pi(t)|^p
 \weq \sum_{i} |f(X_{\xi_i})-f(X_t)|^p \, 1_{(t_{i-1},t_i]}(t).
\]
Because $|t-\xi_i| \le t_i - t_{i-1}$ for $t,\xi_i \in [t_{i-1},t_i]$, we see by applying \eqref{eq:Avikainen2}, the subadditivity of $t \mapsto t^{1-\theta}$, and the formula $\sum_{i=1}^k (t_i - t_{i-1}) = T$ that
\begin{align*}
 \E \int_0^T |h_\pi(t)|^p \, t^{-\theta} dt
 &\weq  \sum_{i} \int_{t_{i-1}}^{t_i} \E  |f(X_{\xi_i})-f(X_t)|^p \, t^{-\theta} dt \\
 &\wle c_2 d_1(f) \sum_{i} (t_i - t_{i-1})^{\changedii{v}} \int_{t_{i-1}}^{t_i} t^{-\theta} dt \\
 &\weq c_2 d_1(f) (1-\theta)^{-1} \sum_{i} (t_i - t_{i-1})^{\changedii{v}} ( t_i^{1-\theta} - t_{i-1}^{1-\theta} ) \\
 &\wle c_2 d_1(f) (1-\theta)^{-1} \sum_{i} (t_i - t_{i-1})^{{\changedii{v}} - \theta + 1} \\
 &\wle c_2 d_1(f) T (1-\theta)^{-1} ||\pi||^{{\changedii{v}} - \theta}.
\end{align*}
We conclude that the expectation of the first term on the right side of \eqref{eq:InterpolationError1} is bounded by
\begin{equation}
\label{eq:InterpolationError2}
 \E \int_0^T |h_\pi(t)|^p \, t^{-\theta} dt
 \wle c_2 d_1(f) T (1-\theta)^{-1} ||\pi||^{{\changedii{v}} - \theta}.
\end{equation}

Let us next analyse the second term in the upper bound of \eqref{eq:InterpolationError1}. By computing integrals part by part along the partition $\pi$, we see that
\[
 \E \left( \int_0^T \int_0^T \frac{|h_\pi(t)-h_\pi(s)|^p}{|t-s|^{p+\theta}} ds dt \right)
 \weq \sum_i m_{i,i} + 2 \sum_{i,j: i<j} m_{i,j},
\]
where
\[
 m_{i,j}
 \weq \E \int_{A_i} \left( \int_{A_j} \frac{| f(X_{\xi_j}) - f(X_t) - f(X_{\xi_i}) + f(X_s)|^p}{|t-s|^{p+\theta}} \, dt \right) ds.
\]

(i) For the diagonal terms, we see by \eqref{eq:Avikainen2} that
\begin{align*}
 m_{i,i}
 \weq \int_{A_i} \int_{A_i} \frac{\E | f(X_t) - f(X_s)|^p}{|t-s|^{p+\theta}} ds dt
 &\weq  \int_{t_{i-1}}^{t_i} \int_{t_{i-1}}^t \frac{\E | f(X_t) - f(X_s)|^p}{(t-s)^{p+\theta}} ds dt \\
 &\wle c_2 d_1(f) \int_{t_{i-1}}^{t_i} \int_{t_{i-1}}^t (t - s)^{{\changedii{v}} - \theta - p} ds dt.
\end{align*} 
Denote $\rho = {\changedii{v}} - \theta - p + 1$. Then $\rho > 0$ and
\[
 \int_{t_{i-1}}^t (t - s)^{\rho-1} ds dt
 \weq \rho^{-1} (1+\rho)^{-1} (t_i - t_{i-1})^{1+\rho}
 \wle \rho^{-1} (1+\rho)^{-1} ||\pi||^\rho (t_i - t_{i-1}),
\]
so that
\[
 m_{i,i}
 \wle c_2 d_1(f) \rho^{-1} (1+\rho)^{-1} ||\pi||^\rho (t_i - t_{i-1}).
\]
By summing over $i$ we find that
\[
 \sum_i m_{i,i}
 \wle c_5' d_1(f) T \, ||\pi||^\rho
\]
with $c_5' = \rho^{-1}  (1+\rho)^{-1} c_2$.

(ii) For the offdiagonal terms with $i < j$, we find that $m_{i,j} \le 2^{p-1} (m_{i,j}(1) + m_{i,j}(2))$ where
\begin{align*}
 m_{i,j}(1)
 &\weq \int_{A_i} \left( \int_{A_j} \frac{\E | f(X_{\xi_i}) - f(X_s)|^p}{|t-s|^{p+\theta}} dt \right) ds, \\
 m_{i,j}(2)
 &\weq \int_{A_i} \left( \int_{A_j} \frac{\E | f(X_{\xi_j}) - f(X_t)|^p}{|t-s|^{p+\theta}} dt \right) ds.
\end{align*}
For the first term, for $i<j$, we find by \eqref{eq:Avikainen2} that
\[
 m_{i,j}(1)
 \wle c_2 d_1(f) ||\pi||^{\changedii{v}} \int_{A_i} \left( \int_{A_j} |t-s|^{-\theta-p} dt \right) ds.
\]
By summing over $j>i$, and then over $i$, we find that
\[
 \sum_i \sum_{j > i} m_{i,j}(1)
 \wle  c_2 d_1(f) ||\pi||^{\changedii{v}} \sum_i \int_{A_i} \left( \sum_{j>i} \int_{A_j} (t-s)^{-\theta-p} dt \right) ds.
\]
Observe next that
\begin{align*}
 \int_{A_i} \left( \sum_{j>i} \int_{A_j} (t-s)^{-\theta-p} dt \right) ds
 &\weq \int_{t_{i-1}}^{t_i} \left( \int_{t_i}^T (t-s)^{-\theta-p} dt \right) ds \\
 &\wle \int_{t_{i-1}}^{t_i} \left( \int_{t_i}^\infty (t-s)^{-\theta-p} dt \right) ds \\
 &\weq \int_{t_{i-1}}^{t_i} (\theta+p-1)^{-1} (t_i-s)^{1-\theta-p} ds \\
 &\weq (\theta+p-1)^{-1} (2-(\theta+p))^{-1}  (t_i-t_{i-1})^{2-(\theta+p)}.
\end{align*}
Hence
\begin{align*}
 \sum_i \sum_{j > i} m_{i,j}(1)
 &\wle c_5'' d_1(f) ||\pi||^{\changedii{v}} \sum_i (t_i-t_{i-1})^{2-(\theta+p)} \\
 &\wle c_5'' d_1(f) T ||\pi||^{1 + {\changedii{v}} - \theta - p} \\
 &\weq c_5'' d_1(f) T ||\pi||^{\rho},
\end{align*}
where $c_5'' = (\theta+p-1)^{-1} (2-(\theta+p))^{-1} c_2$.

For the second term, for $i<j$, we find by \eqref{eq:Avikainen2} in the same way as above that
\[
 m_{i,j}(2)
 \wle c_2 d_1(f) ||\pi||^{\changedii{v}} \int_{A_i} \left( \int_{A_j} |t-s|^{-\theta-p} dt \right) ds.
\]
This upper bound is the same as that earlier, with $i$ and $j$ interchanged. Therefore $\sum_j \sum_{i: i<j} m_{i,j}(2)$ satisfies the same upper bound as above, and we conclude that
\begin{equation}
\label{eq:InterpolationError3}
\begin{aligned}
 \E \left( \int_0^T \int_0^T \frac{|h_\pi(t)-h_\pi(s)|^p}{|t-s|^{p+\theta}} \right) ds dt
 &\wle \sum_i m_{i,i} + 2^{p+1} \sum_{i} \sum_{j:j>i} m_{i,j}(1) \\
 &\wle (c_5' + 2^{p+1} c_5'') T \, d_1(f) ||\pi||^\rho.
\end{aligned}
\end{equation}
By substituting \eqref{eq:InterpolationError2} and \eqref{eq:InterpolationError3} into \eqref{eq:InterpolationError1}, we conclude that 
\[
 ||h_\pi||_{W^\theta_1(0+)}^p
 \wle c_3 c_2 T (1-\theta)^{-1} d_1(f) ||\pi||^{{\changedii{v}} - \theta} + c_4(c_5' + 2^{p+1} c_5'') T \, d_1(f) ||\pi||^\rho.
\]
Because ${\changedii{v}} - \theta = \rho + (p-1) \ge \rho$, we conclude that $||\pi||^{{\changedii{v}} - \theta} \le T ||\pi||^\rho$ for $||\pi|| \le T$. Hence the claim follows.
\end{proof}

\appendix

\section{Approximation by smooth functions}

The purpose of this section is to approximate a function $f$ of locally finite variation by a smooth function $f_n$, and to show that for a sufficiently regular random process $X$, the approximation can be made so that $f_n \circ X$ is with high probability close to $f \circ X$ in the Gagliardo seminorm. We approximate $f$ in a standard manner by
\begin{equation}
 \label{eq:Mollification}
 f_n(x)
 \weq \E f(x - n^{-1}\xi)
\end{equation}
where $\xi$ is a real-valued random variable in $[0,1]$ with an infinitely differentiable probability density function. The following is the main result of this section.

\begin{proposition}
\label{the:SmoothGagliardo}
Let $X$ be a \Holder continuous random process of order $\alpha > 0$ with a probability density function bounded according to Assumption~\ref{ass:Density}. Then for any right-continuous function $f$ of locally finite variation and any $\theta \in (0,\alpha)$,
\[
 [f_n \circ X - f \circ X]_{\theta,1} \ \pto \ 0.
\]
\end{proposition}

The proof of the proposition utilises the following two auxiliary results.

\begin{lemma}
\label{the:SmoothLS}
Let $f$ be bounded, increasing, and right continuous. Then the function $f_n$ defined by \eqref{eq:Mollification} is increasing, infinitely differentiable and bounded, and $f_n \to f$ pointwise. Moreover, the Lebesgue--Stieltjes measure $\mu_{f_n}$ of $f_n$ satisfies
\[
 \int_\R h(y) \, \mu_{f_n}(dy)
 \to 
 \int_\R h(y) \, \mu_{f}(dy)
\]
for all bounded continuous $h$ on $\R$.
\end{lemma}
\begin{proof}
The total mass of $\mu_f$ equals $\mu(\R) = f(+\infty) - f(-\infty)$ and is finite and nonnegative. Lebesgue's dominated convergence theorem implies that $f_n \to f$ pointwise. Lebesgue's monotone convergence theorem implies that $f_n(\pm\infty) = f(\pm\infty)$.

We can write $f_n$ as a convolution
\[
 f_n(x)
 \weq \int_{-\infty}^\infty f(x - y) \phi_n(y) \, dy
\]
where $\phi_n(y) = n \phi(ny)$ is the probability density function of $n^{-1}\xi$. This implies that $f_n$ is infinitely differentiable. The definition also shows that $f_n$ is increasing. Hence the Lebesgue--Stieltjes measure of $f_n$ equals
\[
 \mu_{f_n}(dx) \weq f'_n(x) dx.
\]
Because $f_n(\pm\infty) = f(\pm\infty)$, we see that $\mu_n(\R) = \mu(\R)$. The pointwise convergence $f_n \to f$ implies that
\[
 \mu_n(-\infty,x]
 \weq f_n(x) - f_n(-\infty)
 \weq f_n(x) - f(-\infty)
 \to f(x) - f(-\infty)
 \weq \mu(-\infty,x]
\]
for all real numbers $x$. Assuming that $f$ is not constant, we have $\mu(\R) \in (0,\infty)$, and we find by applying the portmanteau theorem that
$\mu_n(dx)/\mu(\R) \to \mu(dx)/\mu(\R)$ weakly. This implies the claim. If $f$ is constant, the claim is trivial because then $\mu_f = \mu_{f_n} = 0$.
\end{proof}

\begin{lemma}
\label{the:MeanContinuity}
If $X$ is a \Holder continuous random process of order $\alpha > 0$ with a probability density function bounded according to Assumption~\ref{ass:Density}, then for any $\theta \in (0,\alpha)$ and $q \ge \theta/\alpha$, the functions
\[
 y \ \mapsto \
 \E \left\{ (1+[X]_{\alpha,\infty})^{-q} \int_0^T \int_0^t \frac{\ind{X_s < y < X_t}}{(t-s)^{1+\theta}} \, ds dt \right\}
\]
and
\[
 y \ \mapsto \
 \E \left\{ (1+[X]_{\alpha,\infty})^{-q} \int_0^T \int_0^t \frac{\ind{X_t < y < X_s}}{(t-s)^{1+\theta}} \, ds dt \right\}
\]
are bounded and continuous.
\end{lemma}
\begin{proof}
(i) Denote $\phi(y) = \E \Phi(y)$, where
\[
 \Phi(y) \weq (1+[X]_{\alpha,\infty})^{-q} \int_0^T \int_0^t \frac{1_{(X_s,X_t)}(y)}{(t-s)^{1+\theta}} \, ds dt.
\]
By Lemma~\ref{the:HolderSingular},
\[
 \Phi(y)
 \wle \theta^{-1} \frac{[X]_{\alpha,\infty}^{\theta/\alpha}}{(1+[X]_{\alpha,\infty})^{q}} \, \int_0^T |X_t-y|^{-\theta/\alpha} \, dt
 \wle \theta^{-1} \int_0^T |X_t-y|^{-\theta/\alpha} \, dt.
\]
Hence by Lemma~\ref{the:Tikanmaki} and Assumption~\ref{ass:Density}, we see that
\[
 \phi(y) 
 \weq \E \Phi(y)
 \wle \theta^{-1} \int_0^T \left(1 + \frac{2}{1-\theta/\alpha} \hat p_t \right) dt.
\]
The right side above is finite and does not depend on $y$. Hence the function $\phi$ is bounded.

(ii) Let us next verify that $\phi$ is right-continuous at $y \in \R$. Fix $\epsilon > 0$ and observe that
\[
 \ind{X_s < y+\epsilon < X_t} - \ind{X_s < y < X_t}
 \weq \ind{y \le X_s < y + \epsilon < X_t} - \ind{X_s < y < X_t \le y + \epsilon}.
\]
Therefore,
\begin{equation}
 \label{eq:RightLimit}
 \phi(y+\epsilon) - \phi(y)
 \weq \E \int_0^T \int_0^t \Phi_{1,\epsilon}(s,t) \, ds dt - \E \int_0^T \int_0^t \Phi_{2,\epsilon}(s,t) \, ds dt,
\end{equation}
where
\begin{align*}
 \Phi_{1,\epsilon}(s,t)
 &\weq (1+[X]_{\alpha,\infty})^{-q} \, \frac{ \ind{y \le X_s < y + \epsilon < X_t} }{(t-s)^{1+\theta}}, \\
 \Phi_{2,\epsilon}(s,t)
 &\weq (1+[X]_{\alpha,\infty})^{-q} \, \frac{ \ind{X_s < y < X_t \le y + \epsilon} }{(t-s)^{1+\theta}}.
\end{align*}
Assumption~\ref{ass:Density} implies that $\ind{X_s(\omega)=y} = 0$ and $\ind{X_t(\omega)=y} = 0$ for $\pr \otimes \Leb^2$ -almost every $(\omega,s,t)$ in $\Omega \times [0,T]^2$ where $[0,T]^2_< = \{(s,t)\in [0,T]: s < t\}$. Hence, as $\epsilon \to 0$,
\[
 \Phi_{1,\epsilon}(s,t) \ \to \ 0
 \quad \text{and} \quad
 \Phi_{2,\epsilon}(s,t) \ \to \ 0.
\]
for $\pr \otimes \Leb^2$ -almost every $(\omega,s,t)$.

Let us next verify that
\begin{equation}
 \label{eq:LimitPhi1}
 \E \int_0^T \int_0^t \Phi_{1,\epsilon}(s,t) \, ds dt
 \ \to \ 0
\end{equation}
as $\epsilon \to 0$. To do this, choose $p \in (1, \frac{1+\alpha}{1+\theta})$ so small that $1/p \ge 1 + \theta - \alpha q$.
Then
\[
 \Phi_{1,\epsilon}(s,t)^p
 \weq (1+[X]_{\alpha,\infty})^{-\tilde q} \, \frac{ \ind{y \le X_s < y + \epsilon < X_t} }{(t-s)^{1+\tilde\theta}}
 \wle (1+[X]_{\alpha,\infty})^{-\tilde q} \, \frac{ \ind{X_s < y + \epsilon < X_t} }{(t-s)^{1+\tilde\theta}},
\]
where $\tilde q = pq$ and $\tilde \theta = (1+\theta)p - 1$. Our choice of $p$ implies that $\tilde q \ge \tilde \theta/\alpha$ and $\tilde \theta \in (0,\alpha)$. By repeating the argument used in part (i) with $q$ and $\theta$ replaced by $\tilde q$ and $\tilde \theta$, we find that
\[
 \E \int_0^T \int_0^t \Phi_{1,\epsilon}(s,t)^p \, ds dt
 \wle \tilde\theta^{-1} \int_0^T \left(1 + \frac{2}{1-\tilde\theta/\alpha} \hat p_t \right) dt,
\]
where the upper bound on the right is finite and does not depend on $\epsilon$. We may hence conclude that the collection of functions $\Phi_{1,\epsilon}$ indexed by $\epsilon > 0$ is uniformly integrable with respect to $\pr \otimes \Leb^2$, and we obtain~\eqref{eq:LimitPhi1}.

Let us next verify that
\begin{equation}
 \label{eq:LimitPhi2}
 \E \int_0^T \int_0^t \Phi_{2,\epsilon}(s,t) \, ds dt
 \ \to \ 0
\end{equation}
as $\epsilon \to 0$. This is not hard, because
\[
 \Phi_{2,\epsilon}(s,t)
 \wle (1+[X]_{\alpha,\infty})^{-q} \, \frac{ \ind{X_s < y < X_t} }{(t-s)^{1+\theta}}
\]
is valid for all $(\omega,s,t)$. The right side is integrable with respect to $\pr \otimes \Leb^2$, by the argument used in part (i). Hence \eqref{eq:LimitPhi2} follows by Lebesgue's dominated convergence theorem. By substituting the limits \eqref{eq:LimitPhi1}--\eqref{eq:LimitPhi2} into the representation~\eqref{eq:RightLimit}, we may conclude that $y \mapsto \phi(y)$ is right continuous.

(iii) Let us next verify that $\phi$ is left-continuous at $y \in \R$. Fix $\epsilon > 0$ and observe that
\[
 \ind{X_s < y-\epsilon < X_t} - \ind{X_s < y < X_t}
 \weq \ind{X_s < y - \epsilon < X_t \le y} - \ind{y-\epsilon \le X_s < y < X_t}.
\]
Therefore,
\begin{equation}
 \label{eq:LeftLimit}
 \phi(y-\epsilon) - \phi(y)
 \weq \E \int_0^T \int_0^t \Psi_{1,\epsilon}(s,t) \, ds dt - \E \int_0^T \int_0^t \Psi_{2,\epsilon}(s,t) \, ds dt,
\end{equation}
where
\begin{align*}
 \Psi_{1,\epsilon}(s,t)
 &\weq (1+[X]_{\alpha,\infty})^{-q} \, \frac{ \ind{X_s < y - \epsilon < X_t \le y} }{(t-s)^{1+\theta}}, \\
 \Psi_{2,\epsilon}(s,t)
 &\weq (1+[X]_{\alpha,\infty})^{-q} \, \frac{ \ind{y-\epsilon \le X_s < y < X_t} }{(t-s)^{1+\theta}}.
\end{align*}
By the same arguments as in part (ii) of the proof, we see that as $\epsilon \to 0$,
\[
 \Psi_{1,\epsilon}(s,t) \ \to \ 0
 \quad \text{and} \quad
 \Psi_{2,\epsilon}(s,t) \ \to \ 0
\]
for $\pr \otimes \Leb^2$ -almost every $(\omega,s,t)$. We also see that the collection of functions $\Psi_{1,\epsilon}$ indexed by $\epsilon > 0$ is uniformly integrable with respect to $\pr \otimes \Leb^2$, and that $\Psi_{2,\epsilon}(s,t)$ is bounded from above by a $\pr \otimes \Leb^2$ -integrable function. Hence,
\[
 \E \int_0^T \int_0^t \Psi_{1,\epsilon}(s,t) \, ds dt
 \ \to \ 0
 \quad \text{and} \quad
 \E \int_0^T \int_0^t \Psi_{2,\epsilon}(s,t) \, ds dt
 \ \to \ 0,
\]
and the left-continuity of $\phi$ follows from the representation~\eqref{eq:LeftLimit}.

(iv) Let us finally show that the function
\[
 \tilde \phi(y)
 \weq
 \E \left\{ (1+[X]_{\alpha,\infty})^{-q} \int_0^T \int_0^t \frac{\ind{X_t < y < X_s}}{(t-s)^{1+\theta}} \, ds dt \right\}
\]
is bounded and continuous. To do this, define another random process $\tilde X$ by $\tilde X_t = - X_t$. Then $\tilde X$ has $\alpha$-\Holder continuous paths with $[\tilde X]_{\alpha,\infty} = [X]_{\alpha,\infty}$, and we may represent $\tilde \phi$ as
\[
 \tilde \phi(y)
 \weq
 \E \left\{ (1+[\tilde X]_{\alpha,\infty})^{-q} \int_0^T \int_0^t \frac{\ind{\tilde X_s < -y < \tilde X_t}}{(t-s)^{1+\theta}} \, ds dt \right\}.
\]
The boundedness and continuity of $\tilde \phi$ now follow from the boundedness and continuity of $\phi$.
\end{proof}

\begin{proof}[Proof of Proposition~\ref{the:SmoothGagliardo}]
(i) Let us first assume that $f$ right-continuous, increasing, and bounded.
Let $f_n$ be the approximation of $f$ defined by \eqref{eq:Mollification}. Denote
\[
 Z_n(s,t) \weq \frac{f_n(X_t) - f_n(X_s)}{(t-s)^{1+\theta}}
 \quad\text{and}\quad
 Z(s,t) \weq \frac{f(X_t) - f(X_s)}{(t-s)^{1+\theta}}.
\]
Observe that
\[
 \E \int_0^T \int_0^t \frac{|Z_n(s,t)|}{(1+[X]_{\alpha,\infty})^q} ds dt
 \weq \int_{\R} h(y) \, \mu_{f_n}(dy)
\]
where $\mu_{f_n}(dy) = f_n'(y) dy$ is the Lebesgue--Stieltjes measure of $f_n$, and
\[
 h(y)
 \weq \E \left\{ (1+[X]_{\alpha,\infty})^{-q} \int_0^T \int_0^t \frac{\ind{X_s < y < X_t} + \ind{X_t < y < X_s}}{(t-s)^{1+\theta}} \, ds dt \right\}
\]
is continuous by Lemma~\ref{the:MeanContinuity}. Hence, by Lemma~\ref{the:SmoothLS},
\[
 \int_{\R} h(y) \, \mu_{f_n}(dy)
 \ \to \ \int_{\R} h(y) \, \mu_{f}(dy),
\]
from which we conclude that
\[
 \E \int_0^T \int_0^t \frac{|Z_n(s,t)|}{(1+[X]_{\alpha,\infty})^q} ds dt
 \ \to \
  \E \int_0^T \int_0^t \frac{|Z(s,t)|}{(1+[X]_{\alpha,\infty})^q} ds dt.
\]
This observation combined with the fact that $Z_n(s,t) \to Z(s,t)$ for $\pr \otimes \Leb^2$ -almost every $(\omega,s,t)$, allows us to conclude \cite[Lemma 1.32]{Kallenberg_2002} that
\[
 \E \int_0^T \int_0^t \frac{|Z_n(s,t)-Z(s,t)|}{(1+[X]_{\alpha,\infty})^q} ds dt
 \ \to \ 0.
\]
As a consequence,
\[
 (1+[X]_{\alpha,\infty})^{-q} \int_0^T \int_0^t |Z_n(s,t)-Z(s,t)| ds dt
 \ \pto \ 0,
\]
and therefore also
\[
 [f_n \circ X - f \circ X]_{\theta,1}
 \weq 2 \int_0^T \int_0^t |Z_n(s,t)-Z(s,t)| ds dt
 \ \pto \ 0.
\]

(ii) Let us next assume that $f$ is right-continuous and of finite variation. Then we may write $f=f^{(1)}-f^{(2)}$ where $f^{(i)}$ are right-continuous, increasing, and bounded. The mollification of $f$ can be written as $f_n = f^{(1)}_n - f^{(2)}_n$. Hence
\begin{align*}
 [f_n \circ X - f \circ X]_{\theta,1}
 &\weq \left[(f^{(1)}_n \circ X - f^{(1)} \circ X) + (f^{(2)}_n \circ X - f^{(2)} \circ X) \right]_{\theta,1} \\
 &\wle \left[(f^{(1)}_n \circ X - f^{(1)} \circ X) \right]_{\theta,1} +  \left[ (f^{(2)}_n \circ X - f^{(2)} \circ X) \right]_{\theta,1}.
\end{align*}
By part (i), both terms on the right tend to zero in probability. Hence the claim is true also in this case.

(iii) Let us next assume that $f$ is right-continuous and of locally finite variation. Let $A_m = \{||X||_\infty \le m\}$ be the event that the path of $X$ is bounded by $m$. Define $f^{(m)}$ by
\[
 f^{(m)}(y)
 \weq 
 \begin{cases}
  f(-m), & \quad y < -m, \\
  f(y), & \quad |y| \le m, \\
  f(m), &\quad y > m.
 \end{cases}
\]
Then $f^{(m)}$ is right-continuous and of finite variation. Moreover, on the event $A_m$, $f \circ X = f^{(m)} \circ X = f^{(m+1)} \circ X$ and $f_n \circ X = f^{(m+1)}_n \circ X$ for all $n \ge 1$. Hence, as a consequence of part (ii), we see that for any $m$ and any $\epsilon>0$,
\begin{align*}
 \pr( [f_n \circ X - f \circ X]_{\theta,1} > \epsilon, A_m)
 &\weq \pr( [f^{(m+1)}_n \circ X - f^{(m+1)} \circ X]_{\theta,1} > \epsilon, A_m) \\
 &\wle \pr( [f^{(m+1)}_n \circ X - f^{(m+1)} \circ X]_{\theta,1} > \epsilon) \\
 & \ \to \ 0
\end{align*}
as $n \to \infty$. We conclude that
\[
 1_{A_m}  [f_n \circ X - f \circ X]_{\theta,1}
 \ \pto \ 0
\]
for all $m \ge 1$. Because $\pr( \cup_m A_m ) = \pr( ||X||_\infty < \infty) = 1$, it follows (Lemma~\ref{the:ConvergenceInProbability}) that $[f_n \circ X - f \circ X]_{\theta,1} \prto 0$, so the claim holds also in this case.
\end{proof}

\section{Stochastic continuity of finite variation functions}

The following result establishes a stochastic continuity property of a finite variation function $f$ evaluated at two random inputs $X_1$ and $X_2$ which are close to each other in expectation, and such that the probability distribution of $X_1$ admits a density function $p_1$. The result can be seen as an extension of \cite[Theorem 2.4]{Avikainen_2009} into a setting where $p_1$ is not globally bounded.

Below the requirement that $f$ is right-continuous can be relaxed by similar arguments as in the proof of \cite[Theorem 2.4]{Avikainen_2009}. The result stated here suffices for the purposes of this paper.

\begin{proposition}
\label{the:WeakContinuity}
Let $1 \le p \le q < \infty$ and $\epsilon \ge 0$, and assume that $f: \R \to \R$ is a right-continuous function of finite variation such that the restriction of $f$ into $(-\epsilon,\epsilon)$ is Lipschitz continuous. Let $X_1$ be random variable which admits a probability density function $p_1(x)$. Then
\[
 \E| f(X_1) - f(X_2)|^p
 \wle c \, d_1(f) (1+d_2(p_1)) \bigg( \E |X_1-X_2|^q \bigg)^{\frac{1}{1+q}}
\]
for all random variables $X_2$ such that $\E|X_1-X_2|^q \le 1$, where $c= 2^{p-1+q/(1+q)} (1+q)^{1/(1+q)})$,
\[
 d_1(f) \weq
 \begin{cases}
 V(f)^p, &\quad \epsilon = 0, \\
 V(f)^p + \big(\sup_{|x| < \epsilon} |f'(x)| \big)^p, &\quad \epsilon > 0,
 \end{cases}
\]
and
\[
 d_2(p_1) \weq
 \begin{cases}
  \bigg( \sup_{x \in \R} p_{1}(x) \bigg)^{\frac{q}{q+1}}, &\quad \epsilon = 0,\\
  \bigg( 2 \epsilon^{-1} + \sup_{|x| > \epsilon/2} p_{1}(x) \bigg)^{\frac{q}{q+1}}, &\quad \epsilon > 0.
  \end{cases}
\]
\end{proposition}

The proof of Proposition~\ref{the:WeakContinuity} is based on the following two auxiliary results.

\begin{lemma}
\label{the:ConditionalExpectation}
Let $X$ be a real-valued random variable and let $A,B$ be events defined on the same probability space as $X$.
\begin{enumerate}[(i)]
\item \label{ite:ConditionalExpectationIncreasing} If $\phi$ is positive and increasing, then
\[
 \E \phi(X) 1_{A \cap B} \wge \E \phi(X) 1_{A \cap \{X \le b\}}
\]
whenever $\pr(A,B) \ge \pr(A,X \le b)$.
\item \label{ite:ConditionalExpectationDecreasing} If $\phi$ is positive and decreasing, then
\[
 \E \phi(X) 1_{A \cap B} \wge \E \phi(X) 1_{A \cap \{X > b\}}
\]
whenever $\pr(A,B) \ge \pr(A,X > b)$.
\end{enumerate}
The above bounds are valid also when in (i) $\{X \le b\}$ is replaced by $\{X < b\}$, and when in (ii) $\{X > b\}$ is replaced by $\{X \ge b\}$.
\end{lemma}
\begin{proof}
Assume that $\phi$ is positive and increasing. Denote $C = \{X \le b\}$. Then the first claim follows after observing that
\begin{align*}
 \E \phi(X) 1_{A \cap B} - \E \phi(X) 1_{A \cap C}
 &\weq \E \phi(X) 1_{A \cap B \cap C^c} - \E \phi(X) 1_{A \cap B^c \cap C} \\
 &\wge \phi(b) \pr(A, B, C^c) - \phi(b) \pr(A, B^c, C) \\
 &\weq \phi(b) [ \pr(A,B) - \pr(A,C)].
\end{align*}

The second claim follows by observing that when $\phi$ is positive and decreasing, the above inequality is valid for $C = \{X > b\}$.
\end{proof}

\begin{lemma}
\label{the:LipschitzCdf}
Let $X$ be a real-valued random variable with a continuous cumulative distribution function $F$. Then for any $a \in \R$, any $p \in [0,\infty)$, and any real-valued random variable $Y$ defined on the same probability space as $X$,
\[
 \E | 1(X > a) - 1(Y > a) |
 \wle c \, b(a)^{\frac{p}{1+p}} \left( \E |X-Y|^p \right)^{\frac{1}{p+1}},
\]
where
\[
 b(a)
 \ = \ \sup_{x: x \ne a} \frac{|F(x)-F(a)|}{|x-a|}.
\]
and $c = 2^{p/(1+p)} (1+p)^{1/(1+p)}$. Moreover, the same upper bound is also valid for $\E | 1(X \ge a) - 1(Y \ge a) |$.
\end{lemma}
\begin{proof}
Observe that
\begin{equation}
 \label{eq:LipshitzCdf1}
 \E | 1(X > a) - 1(Y > a) |
 \weq \pr(Y \le a < X) + \pr(X \le a < Y).
\end{equation}
We will derive upper bounds for the probabilities on the right, one by one.

Assume first that $\pr(Y \le a < X) < 1$. Then because the cumulative distribution function of $X$ is continuous, we may fix a number $b \in [a,\infty)$ such that $\pr(Y \le a < X) = \pr(a < X \le b)$. Then by applying Lemma~\ref{the:ConditionalExpectation}:\eqref{ite:ConditionalExpectationIncreasing} with $\phi(x) = (x-a)^p 1(x > a)$, $A = \{X > a\}$ and $B = \{Y \le a\}$ we find that
\begin{align*}
 \E |X-Y|^p 1(Y \le a<X)
 &\wge \E (X-a)^p 1(Y \le a < X) \\
 &\wge \E (X-a)^p 1(a < X \le b) \\
 &\weq \E (X-a)^p 1_{(a,b]}(X).
\end{align*}
Let $Q$ be a quantile function of $X$, and note \cite[Section A.3]{Follmer_Schied_2004} that
\begin{align*}
 \E (X-a)^p 1_{(a,b]}(X)
 &\weq \int_{F(a)}^{F(b)} (Q(u)-a)^p \, du \\
 &\weq \int_{F(a)}^{F(b)} \left( \frac{Q(u)-a}{u-F(a)} \right)^p   (u-F(a))^p \, du \\
 &\wge c_a^p \int_{F(a)}^{F(b)} (u-F(a))^p \, du \\
 &\weq c_a^p \, \frac{(F(b)-F(a))^{p+1}}{p+1}
 \weq c_a^p \, \frac{\pr(a < X \le b)^{p+1}}{p+1}
\end{align*}
where
\[
 c_a = \inf_{u \in (F(a),1)} \left( \frac{Q(u)-a}{u-F(a)} \right).
\]
After recalling that $\pr(Y \le a < X) = \pr(a < X \le b)$, we may conclude that
\begin{equation}
 \label{eq:LipshitzCdf2}
 \E |X-Y|^p 1(Y \le a<X)
 \wge
 c_a^p \, \frac{\pr(Y \le a < X)^{p+1}}{p+1}.
\end{equation}
The above inequality is also true in the case with $\pr(Y \le a < X) = 1$, in which case we choose $b=\infty$ and replace $F(b)$ by 1 in the above derivation.

To obtain an upper bound for the second term on the right side of \eqref{eq:LipshitzCdf1}, assume that $\pr(X \le a < Y) < 1$. The by the continuity of the cdf of $X$, we may fix a number $c \le a$ such that $\pr(X \le a < Y) = \pr(c < X \le a)$. Then by applying Lemma~\ref{the:ConditionalExpectation}:\eqref{ite:ConditionalExpectationDecreasing} with $\phi(x) = (a-x)^p 1(x \le a)$, $A = \{X \le a\}$ and $B = \{Y > a\}$ we find that
\begin{align*}
 \E |X-Y|^p 1(X \le a < Y)
 &\wge \E (a-X)^p 1(X \le a < Y) \\
 &\wge \E (a-X)^p 1(c < X \le a) \\
 &\weq \E (a-X)^p 1_{(c,a]}(X).
\end{align*}
A similar computation as above using the quantile function shows that
\begin{align*}
 \E (a-X)^p 1_{(c,a]}(X)
 &\weq \int_{F(c)}^{F(a)} (a-Q(u))^p \, du\\
 &\weq \int_{F(c)}^{F(a)} \left( \frac{a-Q(u)}{F(a)-u} \right)^p (F(a)-u)^p \, du \\
 &\wge \hat c_a^p \int_{F(c)}^{F(a)} (F(a)-u)^p \, du \\
 &\weq \hat c_a^p \, \frac{(F(a)-F(c))^{p+1}}{p+1}
 \weq \hat c_a^p \, \frac{\pr(c < X \le a)^{p+1}}{p+1}
\end{align*}
where
\[
 \hat c_a = \inf_{u \in (0,F(a))} \left( \frac{a-Q(u)}{F(a)-u} \right).
\]
After recalling that $\pr(X \le a < Y) = \pr(c < X \le a)$, we may conclude that
\begin{equation}
 \label{eq:LipshitzCdf3}
 \E |X-Y|^p 1(X \le a<Y)
 \wge
 \hat c_a^p \, \frac{\pr(X \le a < Y)^{p+1}}{p+1}.
\end{equation}
Again, the above inequality is also true in the case with $\pr(X \le a < Y) = 1$, in which case we choose $c=-\infty$ and replace $F(c)$ by 0 in the above derivation.

By combining \eqref{eq:LipshitzCdf2}--\eqref{eq:LipshitzCdf3}, then applying Jensen's inequality $( \frac{x+y}{2} )^{1+p} \le \frac{x^{1+p} + y^{1+p}}{2}$, and then equation \eqref{eq:LipshitzCdf1}, we find that
\begin{align*}
 \E |X-Y|^p
 &\wge c_a^p \, \frac{\pr(Y \le a < X)^{p+1}}{p+1} + \hat c_a^p \, \frac{\pr(X \le a < Y)^{p+1}}{p+1} \\
 &\wge 2^{-p} (c_a \wedge \hat c_a)^p (p+1)^{-1} (\pr(Y \le a < X) + \pr(X \le a < Y))^{p+1} \\
 &\weq 2^{-p} (c_a \wedge \hat c_a)^p (p+1)^{-1} (\E | 1(X>a) - 1(Y>a) | )^{p+1}.
\end{align*}
To conclude the claim, it suffices \cite[Section A.3]{Follmer_Schied_2004} to note that
\[
 c_a \wedge \hat c_a
 \wge \inf_{u \in (0,1): u \ne F(a)} \frac{|a-Q(u)|}{|F(a)-u|}
 \wge \left( \sup_{x: x \ne a} \frac{|F(x)-F(a)|}{|x-a|}\right)^{-1}.
\]

By replicating the same proof with obvious modifications, one can verify that the same upper bound valid also for $\E | 1(X \ge a) - 1(Y \ge a) |$.
\end{proof}

\begin{proof}[Proof of Proposition \ref{the:WeakContinuity}]
Assume first that $\epsilon > 0$. Denote by $\mu$ the variation measure of $f$. Note first that $|f(x_2)-f(x_1)| \le \mu(x_1,x_2]$ for $x_1 \le x_2$ and $|f(x_2)-f(x_1)| \le \mu(x_2,x_1]$ for $x_1 > x_2$. We may write this estimate as
\[
 |f(x_2)-f(x_1)|
 \wle \int_\R [1_{(x_1,x_2]}(y) + 1_{(x_2,x_1]}(y) ] \, \mu(dy).
\]
We will split the above integral into two parts by expressing the above estimate as
\[
 |f(x_2)-f(x_1)|
 \wle I_1(x_1,x_2) + I_2(x_1,x_2),
\]
where 
\[
 I_1(x_1,x_2)
 \weq \int_{(-\epsilon,\epsilon)} [1_{(x_1,x_2]}(y) + 1_{(x_2,x_1]}(y) ] \, \mu(dy)
\]
and
\[
 I_2(x_1,x_2)
 \weq \int_{\R \setminus (-\epsilon,\epsilon)} [1_{(x_1,x_2]}(y) + 1_{(x_2,x_1]}(y) ] \, \mu(dy).
\]
By the Lipschitz continuity assumption we know that $f$ has a derivative $f'$ at almost every point in $(-\epsilon,\epsilon)$ bounded by a constant $c_1(f)$. The variation measure of $f$ satisfies $\mu(dy) = |f'(y)|dy$ on $(-\epsilon,\epsilon)$, so that
\begin{align*}
 I_1(x_1,x_2)
 \wle c_1(f) \int_{-\infty}^\infty [1_{(x_1,x_2]}(y) + 1_{(x_2,x_1]}(y) ] \, dy 
 \weq c_1(f) |x_1-x_2|,
\end{align*}
and
\begin{equation}
 \label{eq:I1Bound}
 \E I_1(X_1,X_2)^p
 \wle c_1(f)^p \, \E |X_1-X_2|^p
 \wle c_1(f)^p \, \left( \E |X_1-X_2|^q \right)^{p/q}.
\end{equation}

For the second term, we see by Jensen's inequality that
\begin{align*}
 I_2(x_1,x_2)^p
 &\wle \mu(\R \setminus (-\epsilon,\epsilon))^{p-1} \int_{\R \setminus [-\epsilon,\epsilon]} [1_{(x_1,x_2]}(y) + 1_{(x_2,x_1]}(y) ]^p \, \mu(dy) \\
 &\wle \mu(\R)^{p-1} \int_{\R \setminus (-\epsilon,\epsilon)} [1_{(x_1,x_2]}(y) + 1_{(x_2,x_1]}(y) ]^p \, \mu(dy) \\
 &\weq V(f)^{p-1} \int_{\R \setminus (-\epsilon,\epsilon)} [1_{(x_1,x_2]}(y) + 1_{(x_2,x_1]}(y) ] \, \mu(dy).
\end{align*}
Observe next that $1_{(x_1,x_2]}(y) + 1_{(x_2,x_1]}(y) = |1(x_2\ge y) - 1(x_1 \ge y)|$. Hence
\[
 \E I_2(X_1,X_2)^p
 \wle V(f)^{p-1} \int_{\R \setminus (-\epsilon,\epsilon)} \E |1(X_2 \ge y) - 1(X_1 \ge y)| \, \mu(dy).
\]
Now by Lemma~\ref{the:LipschitzCdf},

\[
 \E |1(X_2 \ge y) - 1(X_1 \ge y)|
 \wle c_3 \, b(y)^{\frac{q}{1+q}} \left( \E |X_1-X_2|^q \right)^{\frac{1}{q+1}},
\]
where
\[
 b(y)
 \weq \sup_{x: x \ne y} \frac{|F_1(x)-F_1(y)|}{|x-y|},
\]
and $c_3 = 2^{q/(1+q)} (1+q)^{1/(1+q)}$, and $F_1$ is the cumulative distribution function of $X_1$. We conclude that
\[
 \E I_2(X_1,X_2)^p
 \wle c_3 V(f)^p \left( \sup_{|y| \ge \epsilon} b(y) \right)^{\frac{q}{q+1}} \left( \E |X_1-X_2|^q \right)^{\frac{1}{q+1}}.
\]
Next, observe that
\[
 \frac{|F_1(x)-F_1(y)|}{|x-y|}
 \wle 2 \epsilon^{-1}
\]
for all $y$ and all $x$ such that $|x-y| \ge \epsilon/2$, and 
\[
 \frac{|F_1(x)-F_1(y)|}{|x-y|}
 \wle \sup_{|x| > \epsilon/2} p_1(x)
 \ =: \ c_0(p_1)
\]
for all $|y| \ge \epsilon$ and $|x-y| < \epsilon/2$. Hence $b(y) \le 2\epsilon^{-1} + c_0(p_1)$ for all $|y| \ge \epsilon$.
Therefore,
\begin{equation}
 \label{eq:I2Bound}
 \E I_2(X_1,X_2)^p
 \wle c_3 V(f)^p \left( 2\epsilon^{-1} + c_1(p_1) \right)^{\frac{q}{q+1}} \left( \E |X_1-X_2|^q \right)^{\frac{1}{q+1}}.
\end{equation}

Now by noting that
\[
 \E |f(X_1)-f(X_2)|^p
 \wle 2^{p-1} \bigg\{ \E I_1(X_1,X_2)^p + \E I_2(X_1,X_2)^p \bigg\}
\]
and collecting the terms in \eqref{eq:I1Bound} and \eqref{eq:I2Bound} into the constants $c$, $d_1(f)$, and $d_2(p_2)$ in the statement, we obtain the claim for $\epsilon > 0$.

When $\epsilon=0$, we can adapt the above proof easily to obtain the claim also in this case. It suffices to note that $I_1(X_1,X_2)=0$, and $b(y) \le \sup_{x \in \R} p_1(x)$ for all $y \in \R$.

\end{proof}

\section{Elementary probability theory}

The following elementary result, arguably well known to many, is included here for the reader's convenience.
\begin{lemma}
\label{the:ConvergenceInProbability}
Let $A_1,A_2,\dots$ be events such that $\pr(\cup_k A_k) = 1$. Consider random variables $X,X_1,X_2,\dots$ such that for all $k \ge 1$, $1_{A_k} X_n \prto 1_{A_k} X$ as $n \to \infty$. Then $X_n \prto X$.
\end{lemma}
\begin{proof}
Recall that $X_n \prto X$ if and only if every subsequence $\N_0 \subset \N$ has a further subsequence along which $X_n \to X$ almost surely \cite[Lemma 4.2]{Kallenberg_2002}. To prove the claim, let $\N_0$ be an arbitrary subsequence of $\N$. For $k=1$, we may choose a subsequence $\N_1 \subset \N_0$ such that $1_{A_1} X_n \to 1_{A_1} X$ almost surely as $n \to \infty$ along $\N_1$. For $k=2$, may choose a further subsequence $\N_2 \subset \N_1$ such that $1_{A_2} X_n \to 1_{A_2} X$ almost surely along $\N_2$. Proceeding this way, we construct a nested collection $\N_k \subset \N_{k-1} \subset \cdots \subset \N_0$ such that $1_{A_k} X_n  \to 1_{A_k} X$ almost surely as $n \to \infty$ along $\N_k$.

Denote the elements of $\N_k$ by $n_{k,1} < n_{k,2} < \cdots$, and define $n_\ell = n_{\ell,\ell}$. Then for all $k \ge 1$, $n_\ell \in \N_k$ for all $\ell \ge k$, so that $1_{A_k} X_{n_\ell} \to 1_{A_k} X$ almost surely as $\ell \to \infty$. Hence we conclude that $1_A X_{n_\ell}  \to 1_A X$ almost surely where $A = \cup_k A_k$. Because $\pr(A)=1$, we conclude that $\pr(X_{n_\ell} \to X) = \pr(1_A X_{n_\ell} \to 1_A X) = 1$. Hence $\N_0' = \{n_\ell: \ell \ge 1\}$ is a subsequence of $\N_0$ along which $X_n \to X$ almost surely. Hence $X_n \prto X$.
\end{proof}

\section{Auxiliary facts}

The lemma helps shows that on the average, singularities do not matter too much, if $X$ has a bounded probability density function. The following lemma has appeared in \cite[Lemma 3.2]{Tikanmaki_2002}. For the reader's convenience, we include a short proof of the result below.
\begin{lemma}
\label{the:Tikanmaki}
(i) Let $X$ be a real-valued random variable with a bounded probability density function $p_X$. Then for any real number $y$ and any $\alpha \in (0,1)$,
\[
 \E |X-y|^{-\alpha}
 \wle 1 + \frac{2}{1-\alpha} \sup_x p_X(x).
\]
(ii) Let $(X_t)$ be a real-valued random process such that $X_t$ has a bounded density $p_{X_t}$ for every $t$. Then for any measurable function $t \mapsto y_t$,
\[
 \E \int_0^T |X_t - y_t|^{-\alpha} \, dt
 \wle T + \frac{2}{1-\alpha} \int_0^T \sup_x p_{X_t}(x) \, dt.
\]
\end{lemma}
\begin{proof}
(i) Denote $c = \sup_x p_X(x)$ and note that
$
 \E |X-y|^{-\alpha} 1_{\{|X-y| > 1\}} \le 1
$
and
\[
 \E |X-y|^{-\alpha} 1_{\{|X-y| \le 1\}}
 = \int_{-1}^{1} |x|^{-\alpha} p_X(x+y) \, dx
 \le c \int_{-1}^{1} |x|^{-\alpha} \, dx
 = \frac{2c}{1-\alpha}.
\]

(ii) This follows at once from (i). Naturally, we assume in part (ii) that the density of $X_t$ is measurable as a function of $t$.
\end{proof}

\section*{Acknowledgements}

\changedii{We thank Ivan Nourdin and Nuutti Hyvönen for helpful discussions, and the anonymous reviewers for comments and remarks that have helped to improve the presentation of the article.}

\bibliographystyle{abbrv}
\bibliography{bibli_p4}

\end{document}